\documentclass[10pt,a4paper]{amsart}

\usepackage[utf8]{inputenc}
\usepackage[T1]{fontenc}
\usepackage{amsmath}
\usepackage{amsfonts}
\usepackage{ae}
\usepackage{icomma}
\usepackage{color}
\usepackage{graphicx}
\usepackage{bbm}
\usepackage{amsthm}
\usepackage{amssymb}
\usepackage{cancel}
\usepackage[breaklinks,pdfpagelabels=false]{hyperref}
\usepackage{enumitem}
\usepackage{bbm}
\usepackage{verbatim}
\usepackage{subfig}
\usepackage{tikz}
\usetikzlibrary{arrows}
\usepackage{amsrefs}
\usepackage{enumitem}
\usepackage{array}

\newcommand{\abs}[1]{\ensuremath{\left| #1 \right|}}

\newcommand{\defeq}{:=}

\newcommand{\Exp}{\operatorname{Exp}}
\newcommand{\dist}{\operatorname{dist}}

\renewcommand{\epsilon}{\varepsilon}

\newtheorem{theorem}{Theorem}[section]
\newtheorem{lemma}[theorem]{Lemma}
\newtheorem{proposition}[theorem]{Proposition}
\newtheorem{corollary}[theorem]{Corollary}

\theoremstyle{definition}

\newtheorem{remark}[theorem]{Remark}

\newcommand{\Poiss}{\operatorname{Poiss}}

\numberwithin{equation}{section}
\numberwithin{theorem}{section}

\begin{document}
\title{First-passage percolation on Cartesian power graphs}
\author{Anders Martinsson}
\thanks{Research supported by a grant from the Swedish Research Council.}
\address{Department of Mathematical Sciences, Chalmers University Of Technology and University of Gothenburg, 41296 Gothenburg, Sweden}
\email{andemar@chalmers.se}

\begin{abstract}
We consider first-passage percolation on the class of ``high-dimensi-onal'' graphs that can be written as an iterated Cartesian product $G\square G \square \dots \square G$ of some base graph $G$ as the number of factors tends to infinity. We propose a natural asymptotic lower bound on the first-passage time between $(v, v, \dots, v)$ and $(w, w, \dots, w)$ as $n$, the number of factors, tends to infinity, which we call the critical time $t^*_G(v, w)$. Our main result characterizes when this lower bound is sharp as $n\rightarrow\infty$. As a corollary, we are able to determine the limit of the so-called diagonal time-constant in $\mathbb{Z}^n$ as $n\rightarrow\infty$ for a large class of distributions of passage times.
\end{abstract}

\maketitle

\section{Introduction}
For any pair of graphs $H_1=(V_1, E_1)$ and $H_2=(V_2, E_2)$, their Cartesian graph product, denoted by $H_1\square H_2$, is defined as the graph with vertex set equal to the Cartesian product $V_1 \times V_2 = \{(v_1, v_2): v_1\in V_1, v_2\in V_2\}$ and edge set equal to the disjoint union $(E_1\times V_2) \sqcup (V_1 \times E_2)$, where an edge of the form $(v, e)$ or $(e, v)$ is interpreted as an edge between $(v, w_1)$ and $(v, w_2)$ or between $(w_1, v)$ and $(w_2, v)$ respectively, where $w_1$ and $w_2$ denote the end-points of $e$. Furthermore, this edge has the same type (undirected, directed, loop) as $e$. We remark that this product is associative up to graph isomorphisms.

First-passage percolation on a graph $H=(V, E)$ is defined in the following way: For each edge $e$ in $H$, we assign an independent non-negative random weight called its passage time, denoted by $T^F_H(e)$, according to some non-negative random distribution $F$. For each path $\gamma$ in $H$ we define its passage time as $T^F_H(\gamma) = \sum_{e\in \gamma}T^F_H(e)$. Furthermore, for each pair of vertices $v, w$ in $H$ we define the first-passage time from $v$ to $w$ as $T^F_H(v, w) = \inf_{\gamma \text{ from }v\text{ to }w} T^F_H(\gamma)$. Unless states otherwise, we will assume the passage times have standard exponential distribution. For this choice of passage times we suppress the superscript and write $T_H(\cdot)$ and $T_H(\cdot, \cdot)$.

In the standard exponential case, first-passage percolation is closely related to the Richardson model, see for instance Section 3 of \cite{BlairStahn}. This is the continuous-time Markov chain with state space $\{0, 1\}^V$. Each vertex in $H$ is assigned either of the states \emph{healthy}, represented by a $0$, or \emph{infected}, represented by a $1$. A healthy vertex becomes infected at rate equal to its number of infected neighbors, and an infected vertex stays infected forever. It is well-known that if the Richardson model is started with $v$ as the only infected vertex, then it is possible to couple this model to first-passage percolation on $H$ such that $w$ is infected a time $t$ if $T_H(v, w) \leq t$.

In the classical setting of first-passage percolation, $H$ is assumed to be the $d$-dimensional integer lattice $\mathbb{Z}^d$ with the nearest neighbour graph structure for some $d\geq 2$. Despite much research on this topic, there are still many central properties of first-passage percolation on $\mathbb{Z}^d$ which are poorly understood. One avenue to obtain quantitative results in this model has been to consider the limiting behaviour in high dimensions. In \cite{Kesten} (see Chapter 8), Kesten gave estimates for the time constants in $\mathbb{Z}^d$ in directions $(1, 0, 0, \dots 0)$ and $(1, 1, \dots 1)$ as $d\rightarrow\infty$ and used this to show that, for a large class of distributions of passage times, the limit shape is not the Euclidean ball in sufficiently high dimension. In the case of $\Exp(1)$ passage times more precise estimates were given by \cites{Dhar86, Dhar88, CEG11}. A recent paper \cite{AT16} extends this to more general distributions. In \cite{FP93}, Fill and Pemantle proposed the $n$-dimensional hypercube as an alternative high-dimensional graph, and this was subsequently studied in \cites{A89,FP93,BK97,M16}.

In this paper, we will consider first-passage percolation on a generalized class of ``high-dimensional'' graphs, namely sequences of graphs $\{G^n\}_{n=1}^\infty$  where $G^n$ is the $n$:th Cartesian power $G\square \dots \square G$ of some fixed base graph $G$. For any vertex $v\in G$ we let $\bar{v} = (v, v, \dots, v) \in G^n$. Here we allow $G$ to have either a finite or (countably) infinite vertex set, it may have a mixture of directed and undirected edges and it may have multiple edges between the same pair of vertices. In order for sums in our analysis to converge we will always assume that $G$ has bounded degree. As loops do not affect first--passage times, all graphs considered in this paper are, without loss of generality, assumed to be loopless.

For graphs of the form $G^n$, there is a natural lower bound on the first-passage time between two vertices. Let us for now assume that edge passage times are $\Exp(1)$ distributed. This be generalized later. For any graph $H$, and any pair of vertices $v, w\in H$, let $\Gamma_H(v, w)$ denote the set of not necessarily self-avoiding paths from $v$ to $w$ in $G$, and let $\Gamma_G^{sa}(v, w)$ denote the subset consisting of all self-avoiding paths. Here we consider the sequence $\{v\}$ to be a self-avoiding path from $v$ to itself of length $0$. Then, for any $t\geq 0$ we have the union bound
\begin{equation}\label{eq:firstmoment}
\mathbb{P}\left( T_H(v, w) \leq t\right) \leq \sum_{\gamma\in \Gamma_H^{sa}(v, w)} \mathbb{P}\left(T_H(\gamma) \leq t\right).
\end{equation}

The following estimate for the distribution of sums of independent exponential random variables is  well-known. A short proof will be given in Section \ref{sec:basicprop}.
\begin{proposition}\label{prop:sumofexp1}
Let $S_n$ be a sum of $n$ independent $\operatorname{Exp}(1)$ random variables. Then, for any $t\geq 0$, 
\begin{equation}
e^{-t} \frac{t^n}{n!} \leq \mathbb{P}(S_n \leq t) \leq \frac{t^n}{n!}
\end{equation}
\end{proposition}

Using Proposition \ref{prop:sumofexp1} to bound the right-hand side of \eqref{eq:firstmoment}, it is easily seen that
\begin{equation}\label{eq:patmostm}
\mathbb{P}\left( T_H(v, w) \leq t\right) \leq m_H(v, w, t) \defeq \sum_{\gamma \in \Gamma_H(v, w)} \frac{t^{\abs{\gamma}}}{\abs{\gamma}!}
\end{equation}
where $\abs{\gamma}$ denotes the length of $\gamma$, that is, the number of  edges in $\gamma$ counted with multiplicity. Note that this sum now goes over all paths from $v$ to $w$ and not just the self-avoiding ones. For $v=w$, we interpret the term $\frac{t^0}{0!}$ as $1$ also for $t=0$. We will below refer to this function as the generating function of the graph $H$. It can be noted that $m_H(v, w, t)$ is the $vw$:th element of $e^{tA_H}$ where $A_H$ denotes the adjacency matrix of $H$.

When $H$ is a Cartesian product, the generating function can be simplified by the following observation. Again, this will be proven in Section \ref{sec:basicprop}.
\begin{proposition}\label{prop:multiplicativity}
Let $H_1$ and $H_2$ be two graphs with vertices $v_1, w_1 \in H_1$ and $v_2, w_2 \in H_2$. Then
\begin{equation}
m_{H_1 \square H_2}( (v_1, v_2), (w_1, w_2), t) = m_{H_1}(v_1, w_1, t) m_{H_2}(v_2, w_2, t).
\end{equation}
\end{proposition}

Letting $H=G^n$, it follows that
\begin{equation}\label{eq:factorfirstmoment}
\mathbb{P}\left( T_{G^n}((v_1, \dots, v_n), (w_1, \dots, w_n)) \leq t\right) 
\leq  \prod_{i=1}^n m_G(v_i, w_i, t).
\end{equation}
In particular, if we focus on the first-passage time between vertices of the form $\bar{v}$ and $\bar{w}$, we obtain the bound
\begin{equation}\label{eq:mvvvwww}
\mathbb{P}\left( T_{G^n}(\bar{v}, \bar{w}) \leq t\right) \leq m_G(v, w, t)^n.
\end{equation}
As a consequence of this, the critical value of $t$, $t^*=t^*_G(v, w)$, given by the solution to $m_G(v, w, t)=1$, is an asymptotic lower bound on $T_{G^n}(\bar{v}, \bar{w})$ in the sense that
\begin{equation}\label{eq:Tcritlb}
\mathbb{P}\left(T_{G^n}(\bar{v}, \bar{w}) \leq t^*-\varepsilon\right) \rightarrow 0 \text{ as } n\rightarrow\infty
\end{equation}
for any fixed $\varepsilon>0$. It is clear that $t^*_G(v, w)$ exists whenever there is a path from $v$ to $w$ in $G$. Otherwise we consider the critical time to be infinite.

Given this lower bound, it is natural to ask what can be said further about the asymptotics of $T_{G^n}(\bar{v}, \bar{w})$. In particular, under which assumptions on $G$, $v$ and $w$ is it true that $T_{G^n}( \bar{v}, \bar{w} )$ converges to $t^*_G(v, w)$ in some sense as $n\rightarrow\infty$? Indeed, convergence in this way has previously been shown for oriented and unoriented versions of the hypercube (see Subsection \ref{ssec:hcube}). Moreover, it follows from Propositions \ref{prop:sumofexp1} and \ref{prop:selfavoid} below that $m_{G^n}(\bar{v}, \bar{w}, t^*) = 1$ is only a constant factor away from the expected number of self-avoiding paths from $\bar{v}$ to $\bar{w}$ in $G^n$ with passage time at most $t^*$. Nevertheless, it turns out that the question of convergence depends non-trivially on $G$, $v$ and $w$.

While the above analysis is stated for standard exponential weights, heuristically, optimal paths in high dimensional first-passage percolation typically consist of edges with short passage times. Therefore, one would expect that only the density of the distribution near zero ultimately matters for the large $n$ limit of the first-passage time. For any $\rho\in [0, \infty]$, we define
\begin{equation}
\mathcal{C}(\rho) := \{ F \text{ non-negative distribution}: \lim_{x\downarrow 0} F(x)/x = \rho\},
\end{equation}
and
\begin{equation}
\mathcal{C}_{L^1}(\rho) := \{ F \in \mathcal{C}(\rho) : \int_0^\infty x\,dF(x) < \infty \}.
\end{equation}

The main result of this paper is a necessary and sufficient condition for for $T^F_{G^n}(\bar{v}, \bar{w})$ to converge to $t^*_G(v, w)/\rho$ as $n\rightarrow\infty$ for any $F\in \mathcal{C}(\rho)$ for $\rho\in(0, \infty)$. For a given $G$, $v$ and $w$ such that $t^*_G(v, w)<\infty$, we define the function
\begin{equation}\label{eq:fst}
f_G^{vw}(s, t)=\sum_{x, y \in G} m_G(v, x, s) m_G(x, y, t) m_G(y, w, t^*-s-t)  \ln\left( m_G(x, y, t)\right)
\end{equation}
whose domain is the set of all pairs $s, t\geq 0$ such that $s+t\leq t^*$. Here we interpret any term of the form $0 \ln 0$ as $0$. One can show, see Proposition \ref{prop:unifconv}, that this sum is uniformly convergent on the domain of $f_G^{vw}$ for any fixed graph $G$. As a consequence, this function is continuous. Further, one can note that $f_G^{vw}(0, t^*)=f_G^{vw}(s, 0)=0$ for any $0\leq s \leq t^*$.
\begin{theorem}\label{thm:main} Let $\rho\in (0, \infty)$ and $F\in \mathcal{C}(\rho)$. Let $G$ be a bounded degree graph, and let $v$ and $w$ be vertices in $G$ such that there exists a path from $v$ to $w$ in $G$. If $f_G^{vw}(s, t) \leq 0$ for all $s, t \geq 0$ such that $s+t\leq t^*=t^*_G(v, w)$, then
\begin{equation*}
T^F_{G^n}(\bar{v}, \bar{w}) \rightarrow \frac{t^*}{\rho} \text{ in probability as }n\rightarrow\infty.
\end{equation*}
Moreover, if $F\in\mathcal{C}_{L^1}(\rho)$, then convergence holds also in $L^1$. On the other hand, if $f_G^{vw}(s, t)>0$ for some such $s$ and $t$, then there exists a $c=c(G, v, w)>0$ such that
\begin{equation*}
\mathbb{P}\left( T^F_{G^n}(\bar{v}, \bar{w}) > \frac{t^*+c}{\rho} \right)\rightarrow 1 \text{ as } n\rightarrow\infty.
\end{equation*}
\end{theorem}
We remind the reader that $G$ in the above theorem is allowed to have either finite or infinite vertex sets. It may have multiple edges between the same pair of vertices, and may have a mixture of directed and undirected edges, but is, without loss of generality, assumed to be loopless.

We will prove Theorem \ref{thm:main} in two steps. We first show the above results in the case where $F=\Exp(1)$. The following result immediately extends this to all $F\in\mathcal{C}(\rho)$.
\begin{theorem}\label{thm:gendist}
Let $\rho\in (0, \infty)$ and $F\in \mathcal{C}(\rho)$. Let $\{H_n\}_{n=1}^\infty$ be a sequence of bounded-degree graphs (not necessarily uniformly bounded), and let $v_n, w_n$ be vertices on $H_n$. If
\begin{equation*}
T_{H_n}(v_n, w_n) \rightarrow t\text{ in probability as }n\rightarrow\infty,
\end{equation*}
then
\begin{equation*}
T^{F}_{H_n}(v_n, w_n) \rightarrow \frac{t}{\rho} \text{ in probability as }n\rightarrow\infty.
\end{equation*}
Moreover, if $F\in \mathcal{C}_{L^1}(\rho)$, the same statement holds for convergence in $L^1$. On the other hand if $\mathbb{P}(T_{H_n}(v_n, w_n) < t) \rightarrow 0$ as $n\rightarrow\infty$, then $\mathbb{P}(T^{F}_{H_n}(v_n, w_n) < t/\rho - o(1)) \rightarrow 0$ for all $F\in \mathcal{C}(\rho)$.
\end{theorem}

For any explicit graph $G$ and vertices $v, w \in G$, we can at least in principle compute $f_G^{vw}(s, t)$ in order to check whether or not the criterion in Theorem \ref{thm:main} holds. Unfortunately, these calculations are often intractable by hand, and one has to resort to numerical computations.

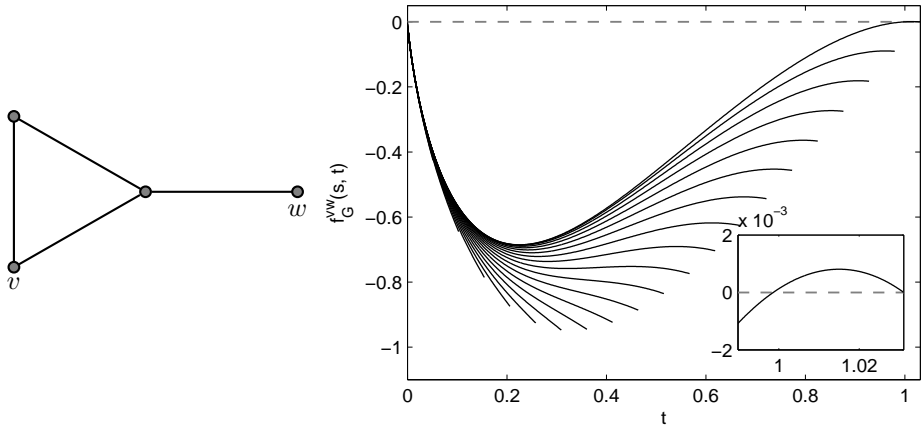
\begin{figure}[h!]
\begin{tabular}{m{0.3\textwidth} m{0.6\textwidth}}
\tikzstyle{every node}=[circle, draw, fill=black!50,
                        inner sep=0pt, minimum width=4pt]
\begin{tikzpicture}[thick,scale=1]
	\draw (0,0) node (a) {}
	-- ++(150:2cm) node (b) {}
	-- ++(270:2cm) node (v) [label=below:$v$] {}
	-- (a) {}
	-- ++(0:2cm) node (w) [label=below:$w$] {};
\end{tikzpicture}
&
\includegraphics[scale=.7]{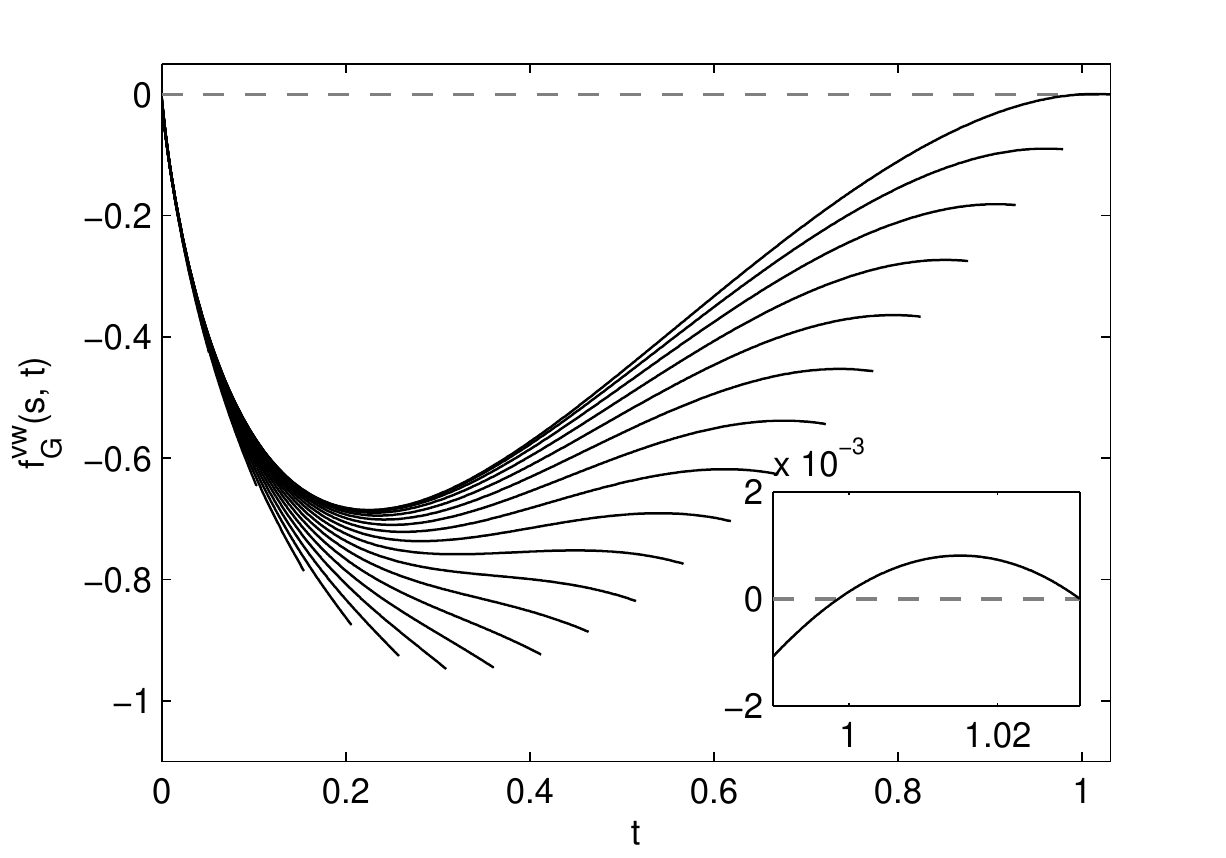}
\end{tabular}
\caption{\label{fig:paw}The smallest example of a simple graph $G$ and vertices $v$, $w$ such that $T_{G^n}(\bar{v}, \bar{w})\not\rightarrow t^*$. In this case we have $t^*\approx 1.03$. The plot on the right shows $f_G^{vw}(s, t)$ as a function of $t$ for different values of $s$. At first glance this function seems non-positive, but upon more careful inspection one sees it attains positive values for small $s$ and $t$ close to $t^*$.}
\end{figure}

By brute-force searching through small graphs, it seems to be fairly common for $f_G^{vw}(s, t)$ to be non-positive for all $s, t$. The smallest example of a simple graph where $f_G^{vw}$ has a positive global maximum is the \emph{paw graph} as shown in Figure \ref{fig:paw}, with $v$ and $w$ as indicated in the figure. It can be noted that the maximum is very close to $0$ ($\approx 0.0008$) in this case, but we believe that this is still sufficiently large not to be an artefact of numerical error. Another noteworthy example is if we take $G$ to be a path graph of length $k$, that is $G^n$ is the $n$-dimensional $(k+1)$-ary hypercube, with $v$ and $w$ as opposite endpoints. Then, we can see numerically that $f_G^{vw}(s, t)$ is non-positive for $k=1, \dots 5$, and appears to attain positive maxima for all $k\geq 6$.

Given the number of spurious graphs for which convergence to $t^*$ holds respectively does not hold, it seems unlikely that the condition in Theorem \ref{thm:main} has some natural reformulation in terms of standard graph properties. Slightly less ambitiously, one might ask if there is some simple sufficient condition on $G$, $v$ and $w$ for convergence. Indeed, at the time of writing, all counter-examples the author has found consist of non-regular graphs. While it remains an open problem whether or not regularity or transitivity is sufficient, we have the following result to this end:

\begin{proposition}\label{prop:charf}
Let $G$ be a graph with bounded degree, and let $v$ and $w$ be fixed vertices in $G$ such that $t^*(v, w)<\infty$. Suppose there is a permutation $\sigma$ of the vertices of $G$ such that, for all vertices $x\in G$, we have
\begin{equation*}
(v, x) \cong (\sigma(x), w)\text{ and }(x, w)\cong (v, \sigma(x)),
\end{equation*}
where $(v_1, v_2) \cong (w_1, w_2)$ denotes that $G$ has a graph automorphism $\varphi$ such that $\varphi(v_1)=w_1$ and $\varphi(v_2)=w_2$. Then $f_G^{vw}(s, t)\leq 0$ for all $s, t$.
\end{proposition}

\begin{corollary}\label{cor:suffG}
Let $G$ be a connected graph with bounded degree, and let $F \in \mathcal{C}(\rho)$ for $\rho\in(0, \infty)$. Each of the following conditions on $G$ is sufficient for $T^F_{G^n}(\bar{v}, \bar{w}) \rightarrow t^*_G(v, w)/\rho$ in probability, and assuming $F \in \mathcal{C}_{L^1}(\rho)$ also in $L^1$, as $n\rightarrow\infty$ for all $v, w \in G$:
\begin{enumerate}[label=\roman*)]
\item $G$ is a Cayley graph of a group $\mathcal{G}$ generated by a finite normal set $S$, that is $\abs{S}<\infty$ and $g S g^{-1} = S\,\forall g\in \mathcal{G}$. In particular, this always holds for Cayley graphs of finitely generated abelian groups.
\item For any pair of vertices $x, y \in G$, we have $(x, y)\cong(y, x)$.
\end{enumerate}
\end{corollary}
\begin{proof}
For $i)$, take $\sigma(x) = v x^{-1} w$. Observe that multiplication by a group element from the left and right both induce graph automorphisms. Hence, $(v, x) \cong (vx^{-1} w, x x^{-1} w) = (\sigma(x), w)$ and $(x, w) \cong (v x^{-1} x, vx^{-1}w) = (v, \sigma(x))$. For $ii)$ we can take $\sigma$ equal to an automorphism that swaps $v$ and $w$. Then, $(v, x) \cong (\sigma(v), \sigma(x)) = (w, \sigma(x))\cong (\sigma(x), w)$ and $(x, w) \cong (\sigma(x), \sigma(w)) = (\sigma(x), v)\cong (v, \sigma(x))$.
\end{proof}


\subsection{Application: The high-dimensional integer lattice}\label{sec:highdim}
Probably the most important case of a Cartesian power graph the $n$-dimensional integer lattice itself. This can be seen to be the $n$:th Cartesian power of $\mathbb{Z}$ with nearest neighbour graph structure. Note that the Cartesian power $\mathbb{Z}^n$ coincides with the usual meaning of this notation.

Studies of first-passage percolation on the high-dimensional integer lattice have focused on estimating the so-called time constant along a coordinate axis, and along a diagonal. Let $F$ be a non-negative distribution with finite expectation. It is well-known, see for instance Section 2.2 of \cite{BlairStahn}, that for any such $F$ and any $n\geq 1$ there exist constants $\mu(F, n)$ and $\mu^*(F, n)$ such that
\begin{align}
\mu(F, n) &= \lim_{k\rightarrow\infty} \frac{T_{\mathbb{Z}^n}(\bar{0}, k \vec{e_1} )}{k},\\
\mu^*(F, n) &= \lim_{k\rightarrow\infty} \frac{T_{\mathbb{Z}^n}(\bar{0}, \bar{k})}{k},
\end{align}
almost surely and in $L^1$, where $\vec{e_1} = (1, 0, 0,\dots, 0)$.

These two quantities were first considered by Kesten (Chapter 8 of \cite{Kesten}, see also \cite{CD83}), who showed that for a large class of distributions $F$ with density $\rho$ near zero, including exponential (where $\rho=1$), we have $\mu(F, n) = \Theta\left(\ln n/(\rho\cdot n) \right)$ and $\mu^*(F, n)=\Theta(1/\rho)$. As a consequence of this, the limit shape is not the Euclidean ball for any such distribution when $n$ is sufficiently large.

In the case of $\Exp(1)$ passage times, the estimate for the time constant along an axis was improved by Dhar \cite{Dhar88}, who showed that
\begin{equation}
\mu(\Exp(1), n) \sim \frac{\ln n}{2 n}\text{ as }n\rightarrow\infty.
\end{equation}
In another, seemingly less known, paper \cite{Dhar86}, Dhar gives a lower bound on the diagonal time constant. This was more recently rediscovered by Couronn\'e, Enriquez and Gerin \cite{CEG11}. For any dimension $n\geq 1$, we have
\begin{equation}
\mu^*(\Exp(1), n) \geq \frac{1}{2} \sqrt{ \alpha_*^2-1 }\approx 0.3313\dots
\end{equation}
where $\alpha_*$ is the unique positive solution to $\coth \alpha = \alpha$.

In a recent paper by Auffinger and Tang \cite{AT16} this was generalized to other distributions. Suppose $F\in\mathcal{C}(\rho)$ for some $\rho \in [0, \infty]$ exists. Then
\begin{equation}
\lim_{n\rightarrow\infty} \frac{n\,\mu(F, n)}{\ln n} = \frac{1}{2\rho},
\end{equation}
and
\begin{equation}\label{eq:ATdiag}
\liminf_{n\rightarrow\infty} \mu^*(F, n)  \geq \frac{1}{2\rho} \sqrt{ \alpha_*^2-1 },
\end{equation}
where $\rho=0$ and $\rho=\infty$ correspond to that the limits are $\infty$ and $0$ respectively. We remark that these results are stated with the technical condition that $F(x)/x = \rho + O\left(1/\abs{\ln x}\right)$ for some neighbourhood $x\in[0, \varepsilon_0]$, but this can be overcome by stochastically sandwiching $F$ between distributions $F_1$ and $F_2$ with constant density $\rho+\varepsilon$ and $\rho-\varepsilon$ respectively near $0$.

We now apply ideas from this paper to show that the bound on the diagonal time constant above is sharp asymptotically in high dimension. By considering the number of paths from $0$ to $k$ in $\mathbb{Z}$ that take precisely $i$ steps to the left, we see that the generating function equals
\begin{equation}\label{eq:mZ}
m_{\mathbb{Z}}(0, k, t) = \sum_{i=0}^\infty {k+2i \choose i} \frac{ t^{k+2i}}{(k+2i)!} = \sum_{i=0}^\infty \frac{t^{k+2i}}{i!(k+i)!}.
\end{equation}
Incidentally, this is equal to $I_k(2t)$, where $I_\alpha(x)$ is the modified Bessel function of the first kind. By Corollary \ref{cor:suffG} we have that for any $F\in\mathcal{C}_{L^1}(\rho)$ where $\rho\in(0, \infty)$, $T^F_{\mathbb{Z}^n}(\bar{0}, \bar{k}) \rightarrow t^*_\mathbb{Z}(0, k)/\rho$ in probability and $L^1$ as $n\rightarrow\infty$.

\begin{theorem}\label{thm:diagmu} Let $F\in \mathcal{C}_{L^1}(\rho)$ where $\rho\in(0, \infty)$. Then
\begin{equation*}
\lim_{n\rightarrow\infty} \mu^*(F, n) = \frac{1}{2\rho} \sqrt{\alpha_*^2-1}.
\end{equation*}
\end{theorem}

\begin{proof}
We start by proving that $t^*_{\mathbb{Z}}(0, k) \sim \frac{1}{2} \sqrt{\alpha_*^2-1}\cdot k$ as $k\rightarrow\infty$. Let $\alpha_i = \frac{k+2i}{k}$. It follows by Stirling's formula that for any $i\geq 1$, we have
\begin{equation}\label{eq:besselstirling}
\frac{ t^{k+2i}}{i!(k+i)!} = \frac{ \Theta(1) }{\sqrt{i(k+i)}}  \left( \frac{ 2et/k}{g(\alpha_i)}\right)^{k+2i},
\end{equation}
where
\begin{equation}
g(\alpha) = (\alpha+1)^{\frac{1}{2}(1+\frac{1}{\alpha})} (\alpha-1)^{\frac{1}{2}(1-\frac{1}{\alpha})}=\exp\left( \ln \sqrt{\alpha^2-1} + \frac{1}{\alpha} \coth^{-1} \alpha\right).
\end{equation}
Moreover, if we continuously extend this function to $\alpha=1$ by letting $g(1)=2$, it follows that \eqref{eq:besselstirling} also holds for $i=0$.

The function $g(\alpha)$ has derivative
\begin{equation}
g'(\alpha) = g(\alpha) \frac{1}{\alpha^2} \left( \alpha - \coth^{-1} \alpha\right),
\end{equation}
and thus attains its global minimum of $e\sqrt{\alpha_*^2-1}$ at $\alpha=\alpha_*$.\footnote{Due to a small misprint in \cite{CEG11}, this point is incorrectly stated to be the global maximum rather than the global minimum. } Hence, by \eqref{eq:mZ} and \eqref{eq:besselstirling} we have that
\begin{equation}
m_{\mathbb{Z}}(0, k, ak) \rightarrow \begin{cases}
0 &\text{ if }  a<\frac{1}{2} \sqrt{\alpha_*^2-1}\\
\infty &\text{ if }  a>\frac{1}{2} \sqrt{\alpha_*^2-1},
\end{cases}
\end{equation}
as $k\rightarrow\infty$, proving $t^*_\mathbb{Z}(0, k) \sim \frac{1}{2}\sqrt{\alpha_*^2-1}\cdot k$.


For any fixed $k$, we have by subadditivity
\begin{equation}
\mu^*(F, n) \leq \mathbb{E}T_{\mathbb{Z}^n}(\bar{0}, \bar{k})/k = t^*_\mathbb{Z}(0, k)/(\rho\,k) + o(1).
\end{equation}
If we take the $\limsup$ of both sides as $n\rightarrow\infty$, and then the limit as $k\rightarrow\infty$, it follows that
\begin{equation}
\limsup_{n\rightarrow\infty} \mu^*(F, n) \leq \lim_{k\rightarrow\infty} t^*_\mathbb{Z}(0, k)/(\rho\,k) = \frac{1}{2\rho} \sqrt{\alpha_*^2-1}.
\end{equation}
The corresponding lower bound is already given by \eqref{eq:ATdiag}.
\end{proof}

\begin{remark}
Cox and Durrett \cite{CD83} considered an oriented version of the diagonal time constant where the base graph $\mathbb{Z}$ is replaced by the doubly infinite directed chain $\vec{\mathbb{Z}}$. They prove that for $\Exp(1)$ weights this time constant converges to a value $\gamma\in [e^{-1}, 2^{-1}]$ as the dimension tends to infinity, and conjecture that the limit is $e^{-1}$. The author has been unable to find any further work on this problem. To end this subsection, we note that one can prove $\gamma = e^{-1}$ in the same way as Theorem \ref{thm:diagmu}. Sharpness of the critical time follows by applying Proposition \ref{prop:charf} to $f^{0k}_{\vec{\mathbb{Z}}}$ with $\sigma(x) = k-x$.
\end{remark}

\subsection{Application: The first-passage time between non-antipodal vertices in the hypercube}\label{ssec:hcube}
The absolutely simplest example of a Cartesian power graph is the $n$-dimensional directed binary hypercube. This is the graph whose vertices are the binary $n$-tuples $\{0, 1\}^n$, and where two vertices are connected by a directed edge if they differ at exactly one position, where the edge is directed towards the vertex with the extra '$1$'. This can be obtained as a power graph $P^n$ by letting $P$ be the graph on two vertices $0$ and $1$ together with a directed edge from $0$ to $1$.

In this case, we have $m_P(0, 1, t) = t$. This implies that for any $t\geq 0$,
\begin{equation}
\mathbb{P}\left( T_{P^n}(\bar{0}, \bar{1}) \leq t\right) \leq t^n,
\end{equation}
and so the critical time $t^*_P(0, 1)=1$ is an asymptotic lower bound on $T_{P^n}(\bar{0}, \bar{1})$ as $n\rightarrow\infty$. This bound was first observed by Aldous \cite{A89} and it was shown by Fill and Pemantle in \cite{FP93} that 
\begin{equation}
T_{P^n}(\bar{0}, \bar{1}) \rightarrow 1\text{ in probability as }n\rightarrow\infty
\end{equation}
by a rather technical second moment argument.

In order to strengthen this result using techniques from this paper, it suffices to observe that $f^{01}_P(s, t) = t\ln t$ which is clearly non-positive for $0\leq t \leq 1$. Hence by Theorem \ref{thm:main}, we have 
\begin{equation}\label{eq:dirhcrho}
T^F_{P^n}(\bar{0}, \bar{1}) \rightarrow 1/\rho
\end{equation}
in probability for any $F\in \mathcal{C}(\rho)$ and in $L^1$ for any $F\in \mathcal{C}_{L^1}(\rho)$ as $n\rightarrow \infty$.

As an aside, one can use this to show a constant upper bound on $T^F_{G^n}(\bar{v}, \bar{w})$ even when $f$ is positive somewhere. Let $G$ be any graph containing vertices $v, w$ such that there exists a path from $v$ to $w$, and let $F\in \mathcal{C}(\rho)$ for some $\rho\in(0, \infty)$. Then, we have
\begin{equation}
T^F_{G^n}(\bar{v}, \bar{w}) \leq \sum_{i=1}^l T^F_{G^n}(\bar{v}_{i-1}, \bar{v}_i) \leq \dist_G(v, w)/\rho + o(1),
\end{equation}
with probability tending to $1$ as $n\rightarrow\infty$, where the last inequality follows from bounding $T^F_{G^n}(\bar{v}_{i-1}, \bar{v}_i)$ by $T^F_{P^n}(\bar{0}, \bar{1})$

As a slightly more complicated example, we have the $n$-dimensional undirected binary hypercube. This is defined in the same way as above, but without assigning directions to the edges. This graph is the $n$:th Cartesian power of $K_2$, the complete graph on two vertices. Let us again denote these vertices by $0$ and $1$.

Since $K_2$ has one path from $0$ to $1$ of length $1$, one of length $3$ and so on, we get 
\begin{equation}
m_{K_2}(0, 1, t) = t + \frac{t^3}{3!}+\frac{t^5}{5!}+\dots = \sinh t.
\end{equation}
Hence, for any $t\geq 0$,
\begin{equation}
\mathbb{P}\left( T_{K_2^n}(\bar{0}, \bar{1}) \leq t\right) \leq \left(\sinh t\right)^n,
\end{equation}
and $t^*_{K_2}(0, 1) = \sinh^{-1}(1) = \ln\left(1+\sqrt{2}\right)$. This bound was first observed by Fill and Pemantle in \cite{FP93} by coupling first-passage percolation to a certain branching process, called the branching translation process, and it was shown by the author in \cite{M16} that 
\begin{equation}
T_{K_2^n}(\bar{0}, \bar{1}) \rightarrow \ln(1+\sqrt{2})
\end{equation}
in probability and $L^p$ norm for any $p\in[1, \infty)$ as $n\rightarrow\infty$ by considering the same process. Using Corollary \ref{cor:suffG} we can immediately extend this result to that 
\begin{equation}
T^{F}_{K_2^n}(\bar{0}, \bar{1})\rightarrow \ln(1+\sqrt{2})/\rho
\end{equation}
in probability for any $F\in\mathcal{C}(\rho)$ and in $L^1$ for any $F\in\mathcal{C}_{L^1}(\rho)$ as $n\rightarrow\infty$.

A generalization of this that appears not to have been studied in this setting before is the Hamming graphs. These are defined as the Cartesian powers $K_q^n$ for some $q\geq 2$, where $K_q$ is the complete graph on $q$ vertices. We denote the vertices of $K_q$ by $0, 1, \dots, q-1$. In this case we get
\begin{equation}
m_{K_q}(0, 1, t) = \frac{1}{q}\left(e^{(q-1)t}-e^{-t}\right).
\end{equation}
This can be seen by computing the matrix exponential 
\begin{equation}
e^{t(E-I)} = e^{tE}\cdot e^{-t I} = (I + E + \frac{q}{2!} E + \frac{q^2}{3!} E + \dots)\cdot e^{-t},
\end{equation}
where $E$ denotes the all ones matrix and $I$ the identity matrix each of size $q\times q$. We note that $\ln (q)/(q-1) \leq t^*_{K_q}(0, 1) \leq \ln (q+1)/(q-1)$, and in particular $t^*_{K_q}(0, 1) \sim \ln(q)/ q$ as $q\rightarrow\infty$. Again, by Corollary \ref{cor:suffG}, $T^{F}_{K_q^n}(\bar{0}, \bar{1})\rightarrow t^*_{K_q}(0, 1)/\rho$ in probability for $F\in \mathcal{C}(\rho)$ and in $L^1$ for $F\in \mathcal{C}_{L^1}(\rho)$ as $n\rightarrow\infty$.

So far we have only considered the first-passage time between diagonal vertices $\bar{v}$ and $\bar{w}$. We will now show a large $n$ limit for the first-passage time between two general vertices in the case of the undirected hypercube. By symmetry of the hypercube, the distribution of the first-passage time between two vertices depends only on their Hamming distance.  Hence, for any $0 \leq k \leq n$, we let $T^F(n, k)$ denote the first-passage time between two vertices at Hamming distance $k$ in $K_2^n$ with respect to distribution $F$. As before, we let $T(n, k) = T^{\Exp(1)}(n, k)$. As we noted above, we have $m_{K_2}(0, 1, t) = m_{K_2}(1, 0, t)=\sinh t$. In the same way we can see that $m_{K_2}(0, 0, t) = m_{K_2}(1, 1, t) = \cosh t$. Hence, by \eqref{eq:factorfirstmoment}, we have
\begin{equation}\label{eq:nonantipod}
\mathbb{P}\left( T(n, k) < t\right) \leq \left( \sinh t \right)^k \left( \cosh t\right)^{n-k},
\end{equation}
for any $t\geq 0$.

For any $0 \leq x \leq 1$, let $\vartheta(x)$ denote the non-negative solution to
\begin{equation}
\left( \sinh \vartheta \right)^{x} \left( \cosh \vartheta\right)^{1-x} = 1.
\end{equation}
It can be seen that $\vartheta(x)$ is continuous and increasing with respect to $x$. We further have $\vartheta(0)=0$, 
 and $\vartheta(1)=\sinh^{-1}1 = \ln(1+\sqrt{2})$.

\begin{theorem}\label{thm:nonantipod} Let $\rho\in(0, \infty)$. For any sequence of integers $\{k_n\}_{n=1}^\infty$ where $0\leq k_n \leq n$ we have that
$$T^F(n, k_n)-\frac{1}{\rho} \vartheta\left(\frac{k_n}{n}\right) \rightarrow 0$$
in probability for any $F\in \mathcal{C}(\rho)$ and in $L^1$ for any $F\in \mathcal{C}_{L^1}(\rho)$ as $n\rightarrow\infty$.
\end{theorem}
We remark that similar statements for related processes have been shown in \cites{M15,L15}.

\begin{lemma}\label{lemma:nonantipodobs} We have the following:
\begin{enumerate}[label=\roman*)]
\item $\mathbb{E} T(n, 1) = O(n^{-1/3})$.
\item $T(n, k)$ is stochastically decreasing in $n$.
\item For any integers $a>0$ and $0\leq b \leq a$, we have $T(an, bn) \rightarrow \vartheta\left(\frac{b}{a}\right)$ in probability and $L^1$-norm as $n\rightarrow\infty.$
\end{enumerate}
\end{lemma}
\begin{proof}
$i)$ For vertices $v$ and $w$ at distance one in $K_2^n$ there is a collection of $n$ edge disjoint paths between $v$ and $w$ of length at most three. Hence $T(n, 1)$ is stochastically smaller than the minimum of $n$ independent $\Gamma(3, 1)$ variables. $ii)$ Consider $K_2^n$ as a subgraph of $K_2^{n+1}$. Then first-passage times in $K_2^n$ dominate those in $K_2^{n+1}$. $iii)$ This follows from Corollary \ref{cor:suffG} by using $K_n^a$ as the base graph and choosing $v$ and $w$ such that $\dist_{K_2^a}(v, w)=b$.
\end{proof}

\begin{proof}[Proof of Theorem \ref{thm:nonantipod}.]
Assume the statement is false for some $F$ and some sequence $\{k_n\}_{n=1}^\infty$. Then there exists a subsequence $\{k_{n_i}\}_{i=1}^\infty$ such that
\begin{equation}
T^F(n_i, k_{n_i}) - \frac{1}{\rho} \vartheta\left(\frac{k_{n_i}}{n_i}\right)
\end{equation}
does not tend to $0$ in the appropriate sense, $\frac{k_{n_i}}{n_i}\rightarrow r\in [0, 1]$ and $k_{n_i}$ is either bounded or tending to infinity. In order to prove that this cannot happen, it suffices by Theorem \ref{thm:gendist} to prove that, for any such sequence,
\begin{equation}
\mathbb{E}\abs{T(n_i, k_{n_i} ) - \vartheta(r) }\rightarrow 0 \text{ as }i\rightarrow\infty,
\end{equation}
with standard exponential passage times.

Assume $r>0$. For any $0< t < \vartheta(r)$, we have by \eqref{eq:nonantipod}
\begin{align*}
&\mathbb{P}\left( T(n_i, k_{n_i}) \leq t \right) \leq \left( \sinh(t)^{k_{n_i}/n_i} \cosh(t)^{1-k_{n_i}/n_i}\right)^{n_i}.
\end{align*}
As $i\rightarrow\infty$, the base in the right-hand side tends to $\sinh(t)^r \cosh(t)^{1-r} < 1$, hence the right-hand side tends to $0$. As this holds for $t$ arbitrarily close to $\vartheta(r)$, we conclude that
\begin{equation}\label{eq:Tnkneg}
\mathbb{P}\left( T(n_i, k_{n_i}) \geq \vartheta(r)-o(1) \right) \rightarrow 1 \text{ as }i\rightarrow\infty.
\end{equation}
Note that as $T(n_i, k_{n_i})$ is non-negative, this holds trivially for $r=0$.

Let $a>0$ by a large integer and let $b$ be the smallest integer such that $\frac{b}{a} \geq \frac{k_{n_i}}{n_i}$ for all but finitely many $i$. Note that $b\leq a$. Then, by subadditivity we have
$$\mathbb{E}T(n_i, k_{n_i}) \leq \mathbb{E}T\left(n_i, b\lfloor \frac{k_{n_i}}{b}\rfloor\right) + \mathbb{E}T\left(n_i, b\left( \frac{k_{n_i}}{b}-\lfloor \frac{k_{n_i}}{b}\rfloor\right)\right).$$
As $0 \leq b\left(\frac{k_{n_i}}{b}-\lfloor \frac{k_{n_i}}{b}\rfloor\right) \leq b$, it follows by Lemma \ref{lemma:nonantipodobs} $i)$ that the second term in the right-hand side is at most $b \mathbb{E}T(n_i, 1) = O(a\,n_i^{-1/3})$ which tends to $0$ as $i\rightarrow\infty$. By the choice of $b$, we know that $a\lfloor\frac{k_{n_i}}{b}\rfloor \leq n_i$ for $i$ sufficiently large, hence by  Lemma \ref{lemma:nonantipodobs} $ii)$,
$$\mathbb{E}T\left(n_i, b\lfloor \frac{k_{n_i}}{b}\rfloor\right) \leq \mathbb{E}T\left(a\lfloor \frac{k_{n_i}}{b}\rfloor, b\lfloor \frac{k_{n_i}}{b}\rfloor\right).$$
Assuming $k_{n_i}\rightarrow\infty$ as $i\rightarrow\infty$, Lemma \ref{lemma:nonantipodobs} $iii)$ implies that the right-hand side tends to $\vartheta\left(\frac{b}{a}\right)$ as $i\rightarrow\infty$. As $\frac{b}{a}$ can be made arbitrarily close to $r$ by taking $a$ sufficiently large, we conclude that
\begin{equation}\label{eq:Tnkexp}
\limsup\limits_{i\rightarrow\infty} \mathbb{E}T(n_i, k_{n_i}) \leq \vartheta(r),
\end{equation}
Note that by Lemma \ref{lemma:nonantipodobs} $i)$ this also holds for $k_{n_i}$ bounded.

By combining \eqref{eq:Tnkneg} and \eqref{eq:Tnkexp}, we conclude that $T(n_i, k_{n_i})$ tends to $\vartheta(r)$ in $L^1$-norm, as desired.
\end{proof}

\subsection{Reading Instructions}
The remainder of the paper is structured as follows: In Section \ref{sec:basicprop}, we will show some properties of the generating function, including Propositions \ref{prop:sumofexp1} and \ref{prop:multiplicativity}, and introduce a conditioned random walk, which will be useful in later sections. The proofs of Theorem \ref{thm:main} in the case of standard exponential distribution and Proposition \ref{prop:charf} are divided between Sections \ref{sec:charf}, \ref{sec:notsharp}, and \ref{sec:sharp}, which may be read independently of each other.

In Section \ref{sec:charf} we prove Proposition \ref{prop:charf}. Section \ref{sec:notsharp} proves that the critical time is not a sharp bound on the first-passage time if $f_G^{vw}(s, t) > 0$ for some $s, t$ in its domain. The proof that $T_{G^n}(\bar{v}, \bar{w}) \rightarrow t^*$ if $f_G^{vw} \leq 0$ is given in Section \ref{sec:sharp}. Finally, Section \ref{sec:gendist} gives a self-sufficient proof of Theorem \ref{thm:gendist}, which implies that Theorem \ref{thm:main} holds for general distributions.

\section{Proofs of basic properties}\label{sec:basicprop}

\begin{proof}[Proof of Proposition \ref{prop:sumofexp1}]
We have
\begin{align*}
\mathbb{P}(S_n \leq t) &= \idotsint\limits_{t_1, \dots, t_n \geq 0} \mathbbm{1}_{t_1+\dots+t_n \leq t} e^{-t_1- \dots - t_n}\,dt_1\dots \,dt_n\\
&\leq \idotsint\limits_{t_1, \dots, t_n \geq 0} \mathbbm{1}_{t_1+\dots+t_n \leq t} \,dt_1\dots\,dt_n\\
&= \idotsint\limits_{0\leq s_1 \leq \dots \leq s_n \leq t} \,ds_1\dots\,ds_n\\
&= \frac{t^n}{n!},
\end{align*}
where, in the second last step, we used the substitution $s_i = t_1+\dots t_i$ for $1\leq i \leq n$. Similarly, we obtain the lower bound
\begin{align*}
\mathbb{P}(S_n \leq t) \geq  e^{-t} \idotsint\limits_{t_1, \dots, t_n \geq 0} \mathbbm{1}_{t_1+\dots+t_n \leq t}\,dt_1\dots \,dt_n = e^{-t} \frac{t^n}{n!}.
\end{align*}
\end{proof}

\begin{proof}[Proof of Proposition \ref{prop:multiplicativity}]
For any path $\gamma \in \Gamma_{H_1\square H_2}( (v_1, v_2), (w_1, w_2) )$, we can note that its projections on the first and second coordinate respectively form paths $\gamma_1 \in \Gamma_{H_1}(v_1, w_1)$ and $\gamma_2 \in \Gamma_{H_2}(v_2, w_2)$. Moreover, for each pair $\gamma_1, \gamma_2$ there are ${\abs{\gamma_1} + \abs{\gamma_2} \choose \abs{\gamma_1}}$ such paths $\gamma$ in $H_1 \square H_2$. This implies that
\begin{align*}
m_{H_1 \square H_2}((v_1, v_2), (w_1, w_2), t) &= \sum_{\gamma \in \Gamma_{H_1 \square H_2}((v_1, v_2), (w_1, w_2))} \frac{t^{\abs{\gamma}}}{\abs{\gamma}!} \\
&= \sum_{\gamma_1 \in \Gamma_{H_1}(v_1, w_1)} \sum_{\gamma_2 \in \Gamma_{H_2}(v_2, w_2)} \frac{ (\abs{\gamma_1}+\abs{\gamma_2} )!}{\abs{\gamma_1}!\abs{\gamma_2}!} \frac{t^{\abs{\gamma_1}+\abs{\gamma_2}}}{(\abs{\gamma_1}+\abs{\gamma_2})!}\\
&= m_{H_1}(v_1, w_1, t) m_{H_2}(v_2, w_2, t).
\end{align*}
\end{proof}

\begin{proposition}\label{prop:convolution}
For any graph $H$ and any vertices $v$, $w$ we have
\begin{equation}
\sum_{x\in H} m_H(v, x, s) m_H(x, w, t) = m_H(v, w, s+t).
\end{equation}
\end{proposition}
\begin{proof}
Note that the concatenation of any path $\gamma_1$ from $v$ to a vertex $x$, and any path $\gamma_2$ from $x$ to $w$ forms a path in $\Gamma_H(v, w)$. Using this, we have
\begin{align*}
\sum_{x\in H} m_H(v, x, s) m_H(x, w, t) &= \sum_{\gamma \in \Gamma_H(v, w)} \sum_{\gamma_1+\gamma_2 = \gamma} \frac{s^{\abs{\gamma_1}}}{\abs{\gamma_1}!} \frac{t^{\abs{\gamma_2}}}{\abs{\gamma_2}!}\\
&=\sum_{\gamma \in \Gamma_H(v, w)} \sum_{k=0}^{\abs{\gamma}} \frac{s^k t^{\abs{\gamma}-k}}{k! (\abs{\gamma}-k)!}\\
&= \sum_{\gamma \in \Gamma_H(v, w)} \frac{1}{\abs{\gamma}!} \sum_{k=0}^{\abs{\gamma}} {\abs{\gamma} \choose k} s^k t^{\abs{\gamma}-k}\\
&=\sum_{\gamma \in \Gamma_H(v, w)} \frac{(s+t)^{\abs{\gamma}}}{\abs{\gamma}!}.
\end{align*}
\end{proof}

\begin{lemma}\label{lemma:mbounds}
Let $H$ be a fixed graph and let $\Delta=\Delta(H)$ denote the maximal degree of any vertex in $H$. Then, for any vertex $v\in H$ and any $t\geq 0$, $\sum_{x \in H} m_H(v, x, t)$ and $\sum_{x\in H} m_H(x, v, t)$ are at most $e^{\Delta t}$. Furthermore,
\begin{equation}
m_H(v, w, t) = \begin{cases} O(t) &\text{ if }v\neq w\\
1+O(t^2) &\text{ if }v=w,\end{cases}
\end{equation}
uniformly over all vertices $v, w$ and $0 \leq t \leq M$ for any $M>0$.
\end{lemma}
\begin{proof}
Since the degree of any vertex in $H$ is at most $\Delta$, then there are at most $\Delta^k$ paths of length $k$ starting at $v$ or ending at $v$ respectively. This implies the bound $1+\Delta t+\frac{\Delta^2 t^2}{2}+\dots = e^{\Delta t}$ for both $\sum_{x} m_H(v, x, t)$ and $\sum_x m_H(x, v, t)$. Since there are no paths of length $0$ from $v$ to $w$ if $v\neq w$, we get $m_H(v, w, t) \leq e^{\Delta t} - 1 = O(t)$. Similarly, since there are no paths from $v$ to $v$ of length $1$, $1 \leq m_H(v, v, t) \leq e^{\Delta t}-\Delta t = 1+O(t^2)$.
\end{proof}

\begin{proposition}\label{prop:unifconv}
For a graph $H$ with distinct vertices $v$ and $w$ such that $t^*_H(v, w)<\infty$, we have, for any $M > 0$ and $\varepsilon>0$, that the sum
\begin{equation}\label{eq:defofg}
g_H^{vw}(s, t, u, \alpha) = \sum_{x, y\in H} m_H(v, x, s) m_H(x, y, t)^{1+\alpha} m_H(y, w, u)
\end{equation}
together with all its term-wise $\alpha$-derivatives converges uniformly for $0\leq s, t, u \leq M$ and $-1+\varepsilon\leq \alpha \leq M$. In particular,
$f_H^{vw}(s, t) = \frac{\partial g_H^{vw} }{\partial \alpha}(s, t, t^*-s-t, 0)$
is continuous, and
\begin{equation}\label{eq:gtaylor}
g_H^{vw}(s, t, t^*-s-t, \alpha) = 1 + \alpha f_H^{vw}(s, t) + O((t^*-t)\alpha^2),
\end{equation}
for any  $s, t\geq 0$ such that $s+t\leq t^*$ and $1-\varepsilon \leq \alpha \leq M$.
\end{proposition}
\begin{proof}
As $m_H(x, y, t)$ is uniformly bounded over all $x, y$ and $0 \leq t \leq M$, it follows that, for any fixed $n\geq 0$, we have $$\frac{\partial^n}{\partial \alpha^n} m_H(x, y, t)^{1+\alpha} = m_H(x, y, t)^{1+\alpha}\left(\ln m_H(x, y, t)\right)^n,$$ bounded in absolute value by a constant $C_n=C_n(M)>0$ for $0\leq t \leq M$ and $-1+\varepsilon \leq \alpha \leq M$. Hence, by Lemma \ref{lemma:mbounds} the sum for $\frac{\partial^n}{\partial \alpha^n} g_H^{vw}(s, t, u, \alpha)$ is dominated by
$$ \sum_{x, y \in H} m_H(v, x, M)\, C_n\, m_H(y, w, M) \leq C_n e^{2\Delta M},$$
thus, we have uniform convergence.

As for \eqref{eq:gtaylor}, it remains to show that
\begin{align*}
&\frac{\partial^2}{\partial\alpha^2} g^{vw}_H(s, t, t^*-s-t, \alpha)\\
&\quad = \sum_{x, y\in H} m_H(v, x, s) m_H(x, y, t)^{1+\alpha} \left(\ln m_H(x, y, t)\right)^2 m_H(y, w, t^*-s-t)
\end{align*}
is $O(t^*-t)$. When $x=v$ and $y=w$, the summand is bounded by a constant times $\left( \ln m_H(v, w, t) \right)^2$. By Taylor expanding $m_H(v, w, t)$ around $t=t^*$, we see that the this is $O( (t^*-t)^2)$. Using Lemma \ref{lemma:mbounds}, the sum of the remaining terms is bounded from above by
\begin{align*}
&C_2 \sum_{x\neq v \text{ or }y \neq w} m_H(v, x, s) m_H(y, w, t^*-s-t)\\
&\qquad = C_2 \sum_{x, y\in H} m_H(v, x, s) m_H(y, w, t^*-s-t)\\
&\qquad\qquad  - C_2m_H(v, v, s)m_H(w, w, t^*-s-t)\\
&\qquad \leq C_2 \left( e^{\Delta s} e^{\Delta(t^*-s-t)} - \left(1+O(s^2) \right)\left(1+O( (t^*-s-t)^2) \right)\right) = O(t^*-t).
\end{align*}
\end{proof}

We will now introduce a kind of random walk that will simplify the proofs of Proposition \ref{prop:charf} and the case in Theorem \ref{thm:main} where $T_{G^n}(\bar{v}, \bar{w}) \rightarrow t^*$. Let $H$ be a graph, and let $v, w$ be vertices in $H$. Let $\Delta_o=\Delta_o(H)$ denote the maximal out-degree of any vertex in $H$. Let $\{\tilde{X}_t\}_{t=0}^\infty$ be the continuous-time random walk on $H$ defined as follows: Initially we have $\tilde{X}=v$. After this, the random walker takes steps at rate $\Delta_o(H)$. Whenever it takes a step, the walker randomly chooses an outgoing edge from $\tilde{X}$, each with probability $1/\Delta_o(H)$, and moves to the opposite end-point of that edge. Note that if the out-degree of $\tilde{X}$ is strictly less than $\Delta_o(H)$, then there is a positive probability of no edge being chosen. If this occurs, we move the walker to an absorbing fail state. Equivalently, we can imagine that we have added directed edges from the vertices of $H$ to a sink so that all vertices have out-degree equal to $\Delta_o$. We define the conditioned random walk from $v$ to $w$ on $H$ in time $s$, $\{X_t\}_{t=0}^s$, as the conditioned process $\{\tilde{X}_t\}_{t=0}^s$ given $\tilde{X}_s=w$.

One important observation to make about this process is that if $\{X_t\}_{t=0}^s$ is the conditioned random walk from $(v_1, v_2)$ to $(w_1, w_2)$ on $H_1\square H_2$ in time $s$, then its coordinates form two independent conditioned random walks from $v_1$ to $w_1$ on $H_1$ and from $v_2$ to $w_2$ on $H_2$ respectively in time $s$. This follows from the coordinate-wise independence of the corresponding process $\{\tilde{X}_t\}_{t=0}^s$. More precisely, in this case the coordinates of $\{\tilde{X}_t\}_{t=0}^s$ take steps independently at rates $\Delta_o(H_1)$ and $\Delta_o(H_2)$ respectively, and the walk fails whenever one of the coordinates fails.

\begin{lemma}\label{lemma:probgogood} For any path $\gamma$ in $H$ starting at $v$, the probability that $\tilde{X}$ traces $\gamma$ during $[0, s]$ is
$$\mathbb{P}\left(\Poiss(\Delta_0 s)=\abs{\gamma}\right) \times \frac{1}{\Delta_o^{\abs{\gamma}}} = e^{-\Delta_o s} \frac{ s^{\abs{\gamma}}}{\abs{\gamma}!}.$$
Moreover, for $t^*=t^*_H(v, w)$, we have 
$$\mathbb{P}\left( \tilde{X}_{t^*}=w \right) = e^{-\Delta_o t^*}.$$
\end{lemma}
\begin{proof} In order for $\tilde{X}$ to trace $\gamma$ during $[0, s]$ it must take precisely $\abs{\gamma}$ steps during this time, and these steps must be the successive edges of $\gamma$, which implies the first statement. Letting $s=t^*$ and summing this over all $\gamma\in \Gamma_H(v, w)$, we get
$$\mathbb{P}\left( \tilde{X}_{t^*} = w\right) = e^{-\Delta_ot^*} m_H(v, w, t^*) = e^{-\Delta_o t^*}.$$
\end{proof}

The following simple consequence of this lemma is useful in considering the conditioned random walk
\begin{proposition}\label{prop:boundcondbyuncond}
Let $\{X_t\}_{t=0}^{t^*}$ be the conditioned random walk from $v$ to $w$ on $H$ in time $t^*$, and let $\{\tilde{X}_t\}_{t=0}^{t^*}$ be the corresponding unconditioned process. Then
$$ \mathbb{P}\left( X \in \cdot\,\right) \leq e^{\Delta_o t^*} \mathbb{P}\left( \tilde{X} \in \cdot\, \right),$$
that is, the probability of any event for $X$ is at most a constant times the corresponding event for $\tilde{X}$.
\end{proposition}
\begin{proof}$\mathbb{P}\left( X \in \cdot\,\right) = \mathbb{P}\left( \tilde{X} \in \cdot\,\middle\vert \tilde{X}_{t^*}=w\right) = \mathbb{P}\left( \tilde{X} \in \cdot\, \wedge \tilde{X}_{t^*}=w\right)/\mathbb{P}\left( \tilde{X}_{t^*}=w\right) = e^{\Delta_o t^*}\mathbb{P}\left( \tilde{X} \in \cdot\, \wedge \tilde{X}_{t^*}=w\right) \leq e^{\Delta_o t^*}\mathbb{P}\left( \tilde{X} \in \cdot\,\right)$.
\end{proof}

\begin{proposition}\label{prop:infprob}
Let $\{X_t\}_{t=0}^{t^*}$ be as above and let $\gamma \in \Gamma_H(v, w)$. Then,
\begin{equation}\label{eq:probtracepath}
\mathbb{P}\left( X\text{ traces }\gamma\right) = \frac{(t^*)^{\abs{\gamma}} }{ \abs{\gamma}!}.
\end{equation}
Furthermore, for any $0 < t_1 < \dots < t_{\abs{\gamma}} < t^*$, the probability that $X$ traces $\gamma$ and takes its steps during $[t_1, t_1+dt_1), [t_2, t_2+dt_2), \dots$ and so on is $dt_1\,dt_2\,\dots\,dt_{\abs{\gamma}}$.
\end{proposition}
\begin{proof}
For any $\gamma \in \Gamma_H(v, w)$, we have by Lemma \ref{lemma:probgogood} that
$$\mathbb{P}\left( X\text{ traces }\gamma\right) = \mathbb{P}\left( \tilde{X}\text{ traces }\gamma\text{ during }[0, t^*]\right)/\mathbb{P}\left(\tilde{X}_{t^*} = w\right) = \frac{(t^*)^{\abs{\gamma}} }{ \abs{\gamma}!}.$$
Furthermore, if we condition on $X$ tracing $\gamma$, then the step times are distributed as a Poisson process with rate $\Delta_o$ conditioned on the number of arrivals in $[0, t^*]$ being equal to $\abs{\gamma}$, that is as  the order statistics of a $\abs{\gamma}$-sample following the uniform law on $[0, t^*]$.
Hence the probability for the steps to occur during the above intervals is $\abs{\gamma}! (t^*)^{-\abs{\gamma}}\,dt_1\,dt_2\dots\,dt_{\abs{\gamma}}$.
\end{proof}

Finally, we show a useful connection between the conditioned random walks and certain sums containing the generating function. One can observe that, by Proposition \ref{prop:convolution}, for any $0\leq t_1 \leq \dots \leq t_k \leq t^*=t^*_H(v, w)$, we have
\begin{align*}
&\sum_{x_1, \dots, x_k \in H} m_H(v, x_1, t_1) m_H(x_1, x_2, t_2-t_1)\cdot \dots \cdot m_H(x_k, w, t^*-t_k)\\
&\qquad = m_H(v, w, t^*)= 1,
\end{align*}
hence this sum defines a probability distribution on $x_1, \dots, x_k$.

\begin{proposition}\label{prop:convisprob}
Let $\{X_t\}_{t=0}^{t^*}$ be as above, and let $0 \leq t_1 \leq \dots \leq t_k \leq t^*$ and $x_1, \dots, x_k \in H$. Then
\begin{equation}\label{eq:convisprob}
\begin{split}
&\mathbb{P}\left( X_{t_i}=x_i \text{ for }1\leq i \leq k\right)\\
&\qquad = m_H(v, x_1, t_1) m_H(x_1, x_2, t_2-t_1)\cdot \dots \cdot m_H(x_k, w, t^*-t_k).
\end{split}
\end{equation}
As a consequence
\begin{equation}\label{eq:fisprob}
f_H^{vw}(s, t) = \mathbb{E}\ln m_H( X_s, X_{s+t}, t).
\end{equation}
\end{proposition}
\begin{proof}
Let us, to simplify notation, write $x_0=v$, $x_{k+1}=w$, $t_0=0$ and $t_{k+1}=t^*$. For any $0 \leq i \leq k$, we have
\begin{align*}
\mathbb{P}\left( \tilde{X}_{t_{i+1}} = x_{i+1} \middle\vert \tilde{X}_{t_i}=x_i\right) &= \sum_{\gamma \in \Gamma_H(x_i, x_{i+1})} e^{-\Delta_o(H) (t_{i+1}-t_i) } \frac{ (t_{i+1}-t_i)^{\abs{\gamma}}  }{ {\abs{\gamma}}! }\\
&=e^{-\Delta_o(H)(t_{i+1}-t_i) } m_H(x_i, x_{i+1}, t_{i+1}-t_i).
\end{align*}
The proposition follows by multiplying these and using Lemma \ref{lemma:probgogood}.
\end{proof}

\section{Proof of Proposition \ref{prop:charf}}\label{sec:charf}
Throughout this section, $G$ denotes a graph containing distinct vertices $v$ and $w$ such that $t^*=t^*_G(v, w)<\infty$. We assume that there exists a permutation $\sigma$ of the (possibly countably infinite) vertex set of $G$ such that $(v, x)\cong (\sigma(x), w)$ and $(x, w)\cong (v, \sigma(x))$ for all $x\in G$. 

\begin{lemma}\label{lemma:indepofs}
For any $G$, $v$ and $w$ as above, the function $f_G^{vw}(s, t)$ as defined in \eqref{eq:fst} does not depend on $s$.
\end{lemma}
\begin{proof}
Writing $x'=\sigma(x)$ and $y'=\sigma(y)$, we have by the second isomorphism that, for any $x\in G$,
\begin{align*}
&\sum_{y\in G} m_G(x, y, t)\ln m_G(x, y, t) m_G(y, w, t^*-s-t)\\
&\qquad = \sum_{y'\in G} m_G(v, y', t)\ln m_G(v, y', t) m_G(y', x', t^*-s-t)
\end{align*}
Hence
\begin{align*}
f_G^{vw}(s, t)&= \sum_{x', y'\in G} m_G(v, y', t)\ln m_G(v, y', t) m_G(y', x', t^*-s-t)m_G(x', w, s)\\
&=\sum_{y'\in G} m_G(v, y', t)\ln m_G(v, y', t) m_G(y', w, t^*-t),
\end{align*}
where the last step follows from Proposition \ref{prop:convolution}.
\end{proof}

As $f_G^{vw}(s, t)$ does not depend on $s$, we now write $f_G^{vw}(t)$. What remains to show is that independence of $s$ implies that $f_G^{vw}(t)$ is convex in $t$. As $f_G^{vw}(0)=f_G^{vw}(t^*)=0$, this will imply  that $f_G^{vw}\leq 0$, as desired. We do this by interpreting this function in terms of entropy of the conditioned random walk. For a random variable $X$ taking values in a finite or countable set, its entropy is given by
\begin{equation}
\mathbb{H}(X) = -\sum_x \mathbb{P}\left( X=x\right) \ln \mathbb{P}(X=x).
\end{equation} 
Furthermore, if $Y$ is another random variable defined on the same probability space as $X$, the conditional entropy of $X$ given $Y$ is
\begin{equation}
\mathbb{H}\left( X \middle\vert Y\right) 
 = \mathbb{H}\left( X, Y\right) - \mathbb{H}\left( Y \right),
\end{equation}
where $\mathbb{H}(X, Y)$ is the entropy of the joint random variable $(X, Y)$. See for instance Lemma 15.7.1 in \cite{AS08} for proofs of standard properties of entropy.

Let $\{X_t\}_{t=0}^{t^*}$ be the conditioned random walk on $G$ from $v$ to $w$ in time $t^*$ as defined in Section \ref{sec:basicprop}.  By Proposition \ref{prop:convisprob}, we have for any $0\leq a \leq b \leq t^*$ that
\begin{equation}
\mathbb{H}(X_b \vert X_a) = f_G^{vw}(t^*-a) - f_G^{vw}(b-a) - f_G^{vw}(t^*-b).
\end{equation}
As the conditioned random walk is Markovian, $\mathbb{H}(X_b\vert X_a)$ is decreasing in $a$ for $a\leq b$ and $b$ fixed. Hence, for any $t, \delta \geq 0$ such that $0 \leq t-\delta \leq t+\delta \leq t^*$, we have
\begin{align*}
0 &\leq \mathbb{H}(X_{t^*-\delta} \vert X_{t^*-t-\delta} ) - \mathbb{H}(X_{t^*-\delta} \vert X_{t^*-t} )\\
&= f_G^{vw}(t-\delta) + f_G^{vw}(t+\delta) - 2 f_G^{vw}(t).
\end{align*}
It follows that $f_G^{vw}(t)$ is midpoint-convex, hence, as it is continuous, it is convex, as desired.


\begin{remark}\label{rem:strongconvexity} For any $G, v$ and $w$ as above it can be noted that
\begin{equation}
f_G^{vw}(\frac{t^*}{2} ) = - \frac{1}{2} \mathbb{H}\left( X_{\frac{t^*}{2}} \right),
\end{equation}
which is strictly negative as $X_{\frac{t^*}{2}}$ has non-trivial distribution. Combining this with $f_G^{vw}(0) = f_G^{vw}(t^*) = 0$ and convexity it follows that $f_G^{vw}(t) \leq -\Omega( t(t^*-t) )$ for $0\leq t \leq t^*$ in this case.
\end{remark}

\section{The critical time is not always sharp}\label{sec:notsharp}
The aim of this section is to prove the second part of Theorem \ref{thm:main} in the case of the standard exponential distribution. Our strategy will be based on the following observation:
\begin{proposition}\label{prop:propconv}
Let $H$ be a graph. For any vertices $v, w \in H$ and $s, t\geq 0$, we have
\begin{equation}
\mathbb{P}(T_H(v, w) \leq s+t) \leq \sum_{x\in H} \mathbb{P}(T_H(v, x) \leq s) \mathbb{P}(T_H(x, w) \leq t)
\end{equation}
\end{proposition}
Note that the event that $T_H(v, w) \leq s+t$ is not contained in the event that there exists an $x$ such that $T_H(v, x)\leq s$ and $T_H(x, w) \leq t$.
\begin{proof}
Run the Richardson model with initial infected vertex $v$ for $s$ time. Consider each existing infectious particle at this time as a distinct type, labelled by its location, and run the model for additional time $t$. For each vertex $x\in H$, $w$ has a type $x$ infection at time $s+t$ if $x$ is infected at time $s$ and if this infection spreads to $w$ during the remaining time $t$. The probability of the former is clearly $\mathbb{P}(T_H(v, x) \leq s)$, and given the former the probability of the latter is at most $\mathbb{P}( T_H(x, w) \leq t)$. Note the inequality, because now the infection competes with other types.
\end{proof}

As a direct consequence of Proposition \ref{prop:propconv}, we have for any $s, t, u \geq 0$ that
\begin{equation}\label{eq:goodbound}
\begin{split}
&\mathbb{P}(T_{G^n}(\bar{v}, \bar{w}) \leq s+t+u )\\ &\qquad\leq 
\sum_{x, y\in G^n} \mathbb{P}(T_{G^n}(\bar{v}, x) \leq s ) \mathbb{P}(T_{G^n}(x, y) \leq t ) \mathbb{P}(T_{G^n}(y, \bar{w}) \leq u ).
\end{split}
\end{equation}
A natural way to bound this sum is to replace each factor in the right-hand side by the corresponding value of the generating function according to \eqref{eq:patmostm}, that is
\begin{equation}\label{eq:badbound}
\mathbb{P}(T_{G^n}(\bar{v}, \bar{w}) \leq s+t+u ) \leq \sum_{x, y\in G^n} m_{G^n}(\bar{v}, x, s) m_{G^n}(x, y, t) m_{G^n}(y, \bar{w}, u).
\end{equation}
However, according to Proposition \ref{prop:convolution}, this is equal to $m_{G^n}(\bar{v}, \bar{w}, s+t+u)$, so this bound in itself gives no improvement on \eqref{eq:mvvvwww}. On the other hand, what can happen and which then allows us to derive a better bound on the first-passage time is that for certain choices of $G, v, w$ most of the contribution to the right-hand side of \eqref{eq:badbound} comes from terms where $m_{G^n}(x, y, t)$ is much larger than one. This motivates us to instead use the bound
$$\mathbb{P}\left(T_{G^n}(x, y) \leq t \right) \leq \mathbb{P}\left(T_{G^n}(x, y) \leq t \right)^{1-\alpha} \leq m_{G^n}(x, y, t)^{1-\alpha}$$ for $\alpha>0$ sufficiently small. By Proposition \ref{prop:multiplicativity}, this implies that
\begin{align*}
\mathbb{P}(T_{G^n}(\bar{v}, \bar{w}) \leq s+t+u ) &\leq 
\sum_{x, y\in G^n} m_{G^n}(\bar{v}, x, s) m_{G^n}(x, y, t)^{1-\alpha} m_{G^n}(y, \bar{w}, u )\\
&= \left( \sum_{x, y \in G} m_G(v, x, s) m_G(x, y, t)^{1-\alpha} m_G(y, w, u) \right)^n.
\end{align*}
Pick $s$ and $t$ such that $f_G^{vw}(s, t)>0$ and let $u=t^*-s-t$. By Proposition \ref{prop:unifconv} and in particular \eqref{eq:gtaylor}, we have
\begin{equation}\label{eq:taylorofm2}
\sum_{x, y \in G} m_G(v, x, s) m_G(x, y, t)^{1-\alpha} m_G(y, w, t^*-s-t) = 1 - \alpha f_G^{vw}(s, t) + O(\alpha^2).
\end{equation}
Picking $\alpha$ sufficiently small so that this expression is strictly less than one it follows by continuity of this sum, as shown in Proposition \ref{prop:unifconv}, that there exists a $c>0$ such that $$\sum_{x, y \in G} m_G(v, x, s) m_G(x, y, t)^{1-\alpha} m_G(y, w, u) < 1$$ for any $t^*-s-t \leq u \leq t^*-s-t+c$. Hence $\mathbb{P}(T_{G^n}(\bar{v}, \bar{w}) \leq t^*+c)\rightarrow 0$ as $n\rightarrow\infty$.


\section{The case where the critical time is sharp}\label{sec:sharp}
Let $G$ be a graph with distinct vertices $v$ and $w$ such that $t^*=t^*_G(v, w) < \infty$. The aim of this section is to prove the first part of Theorem \ref{thm:main} in the case of the standard exponential distribution.

The proof strategy is divided into two steps. We first show that if $f^{vw}_G(s, t) \leq 0$ everywhere, then $\mathbb{P}(T_{G^n}(\bar{v}, \bar{w}) \leq t^* + o(1))$ is bounded away from $0$ as $n\rightarrow\infty$. Second, we show that if, for some $t\geq 0$, $\mathbb{P}( T_{G^n}(\bar{v}, \bar{w} ) \leq t )$ is bounded away from zero as $n\rightarrow\infty$, then $\mathbb{E}T_{G^n}(\bar{v}, \bar{w}) \leq t+o(1)$. As $T_{G^n}(\bar{v}, \bar{w})$ is non-negative and asymptotically almost surely at least $t^*-o(1)$, this implies that $T_{G^n}(\bar{v}, \bar{w})\rightarrow t^*$ in probability and $L^1$ as $n\rightarrow\infty$, as desired.

Subsection \ref{sec:weakbound} shows the first step of this proof under the additional assumption that $f_G^{vw}(s, t)=-\Omega(t(t^*-t))$. As noted in Remark \ref{rem:strongconvexity}, this stronger assumption holds for all $G, v, w$ that satisfy Proposition \ref{prop:charf} and hence also Corollary \ref{cor:suffG}. In Subsection \ref{sec:recineq} we present a bootstrap argument which shows the second step of the proof. Finally, in Subsection \ref{sec:weightedgraphs} we show how the assumption of $f_G^{vw}(s, t)=-\Omega(t(t^*-t))$ in the first subsection can be relaxed to $f_G^{vw} \leq 0$ by considering a generalization to certain weighted graphs.

\subsection{A positive probability upper bound}\label{sec:weakbound}

We will begin by proving a lower bound on the probability that $T_H(v, w)$ is at most $t^*_H(v, w)$ for any graph $H$ and vertices $v$ and $w$ in terms of the conditioned random walk on $H$.
\begin{proposition}\label{prop:Tubrw}
Let $H$ be a graph with distinct vertices $v$ and $w$ such that $t^*_H(v, w)<\infty$. Let $\{X_t\}_{t=0}^{t^*}$ be a conditioned random walk from $v$ to $w$ on $H$ in time $t^*=t^*_H(v, w)$. Let $L$ denote the number of steps the walker takes, and for any $1\leq i \leq L$, let $T_i$ denote the time of the $i$:th jump. To simplify notation we write $T_0=0$. Let
\begin{equation}
C(X) = \sum_{0\leq i<j\leq L} m_H(X_{T_i}, X_{T_j}, T_j-T_i).
\end{equation}
Then,
\begin{equation}\label{eq:Tubrw}
\mathbb{P}\left( T_H(v, w) \leq t^* \right) \geq \mathbb{E}\left[ \mathbbm{1}_{X \text{ is self-avoiding}}\,e^{-C(X)}\right].
\end{equation}
\end{proposition}
\begin{proof}
A self-avoiding path $\gamma$ from $v$ to $w$ in $H$ is a geodesic if $T_H(\gamma) = T_H(v, w)$. We say that $\gamma$ is the unique geodesic from $v$ to $w$ if $T_H(\gamma') > T_H(v, w)$ for all $\gamma' \in \Gamma^{as}_H(v, w)\setminus\{\gamma\}$.

Pick any $\gamma \in \Gamma^{sa}_H(v, w)$. We denote the vertices along $\gamma$ by $v=v_0, v_1, \dots, v_{\abs{\gamma}}=w$ and the edges by $e_1, \dots, e_{\abs{\gamma}}$. Let $t_1, \dots, t_{\abs{\gamma}} \geq 0$ such that $t_1+\dots +t_{\abs{\gamma}} \leq t^*$ and condition on the event that $T_H(e_i) = t_i$ for all $1\leq i \leq \abs{\gamma}$. We derive a lower bound on the probability that $\gamma$ is the unique geodesic from $v$ to $w$.

By some straight-forward combinatorial reasoning, one sees that if $\gamma$ is not a unique geodesic, then there must exist a shortcut $\gamma'$ of $\gamma$, that is, a self-avoiding path $\gamma'$ from $v_i$ to $v_j$ for some $i<j$ such that $\gamma'$ shares no edges with $\gamma$ and $T_H(\gamma')\leq t_{i+1}+\dots+t_j$. Note that by this definition a shortcut need not be strictly faster than the corresponding segment in $\gamma$, just as fast, though the probability that equality holds for any $\gamma'$ is clearly $0$. The event that a certain path $\gamma'$ is not a shortcut is clearly increasing with respect to the passage times of edges in $H$ which are not in $\gamma$. Hence, by the FKG inequality, see Subsection 2.2 in \cite{Grimmett},
\begin{equation}\label{eq:gammauniquegeo}
\begin{split}
&\mathbb{P}\left( \gamma \text{ unique geodesic}\middle\vert T_H(e_i)=t_i, 1\leq i \leq \abs{\gamma}\right)\\
&\qquad \geq \prod_{\gamma'}\left(1-\mathbb{P}\left(\gamma'\text{ shortcut}\middle\vert T_H(e_i)=t_i, 1\leq i \leq \abs{\gamma}\right)\right).
\end{split}
\end{equation}
We remark that \cite{Grimmett} only states this inequality for finite collections of events, but this immediately extends to countable collections as here by continuity of the probability measure.

Since $m_H(x, y, t)$ is an upper bound on the expected number of self-avoiding paths from $x$ to $y$ with passage time at most $t$, it follows that
\begin{equation}\label{eq:expshortcuts}
\begin{split}
&\sum_{\gamma'} \mathbb{P}\left( \gamma'\text{ shortcut}\middle\vert T_H(e_i)=t_i, 1\leq i \leq \abs{\gamma}\right)\\
&\qquad \leq \sum_{0\leq i < j \leq \abs{\gamma}} m_H(v_i, v_j, t_{i+1}+\dots +t_j) := C_H(\gamma, t_1, \dots, t_{\abs{\gamma}}).
\end{split}
\end{equation}

Now, for any specific $\gamma'$, it is easily seen that the probability of it being a shortcut is at most the probability that an $\Exp(1)$ random variable is at most $t^*$, which is $1-e^{-t^*}$. Using the elementary inequality $1-p \geq e^{\frac{\ln(1-\varepsilon)}{\varepsilon} p}$ for $0 \leq p \leq \varepsilon$, it follows from \eqref{eq:gammauniquegeo} and \eqref{eq:expshortcuts} that
\begin{equation}
\begin{split}
&\mathbb{P}\left( \gamma \text{ unique geodesic}\middle\vert T_H(e_i)=t_i, 1\leq i \leq \abs{\gamma}\right)\\
&\qquad \geq \exp\left( -\frac{t^*}{1-e^{-t^*}} C_H(\gamma, t_1, \dots, t_{\abs{\gamma}}) \right).
\end{split}
\end{equation} 
Taking the expected value over the passage times along $\gamma$ and summing over all self-avoiding paths from $v$ to $w$, we obtain
\begin{equation}\label{eq:upperboundstrangeconstants}
\begin{split}
&\mathbb{P}\left( T_H(v, w) \leq t^*\right)\\
&\qquad \geq \sum_{\gamma \in \Gamma^{sa}_H(v, w)} \idotsint\limits_{\sum_i t_i \leq t^*} e^{- \frac{t^*}{1-e^{-t^*}} C_H(\gamma, t_1, \dots, t_{\abs{\gamma}}) -t_1-\dots -t_{\abs{\gamma}}}\,dt_1\dots\,dt_{\abs{\gamma}}.
\end{split}
\end{equation}

Let us now use a small trick in order to simplify this expression. Let $k$ be a large integer, and define $H'$ as the graph with the same vertex set as $H$ and which contains $k$ copies of each edge in $H$. It is easily seen that:
\begin{enumerate}[label=\roman*)]
\item $T_H(v, w)$ has the same distribution as $k T_{H'}(v, w)$
\item for each path $\gamma$ in $\Gamma_H(x, y)$ there are $k^{\abs{\gamma}}$ copies in $\Gamma_{H'}(x, y)$
\item $m_H(x, y, t) = m_{H'}(x, y, t/k)$
\item $t^*_{H'}(v, w) = t^*/k$.
\end{enumerate}
Applying \eqref{eq:upperboundstrangeconstants} to $H'$, it follows that
\begin{align*}
&\mathbb{P}\left( T_H(v, w) \leq t^* \right)\\
&\geq \sum_{\gamma \in \Gamma^{sa}_{H'}(v, w)}\; \idotsint\limits_{\sum_i t_i \leq t^*/k} e^{ -\frac{ t^*/k }{ 1- e^{-t^*/k}} C_{H'}(\gamma, t_1, \dots, t_{\abs{\gamma}})-t_1-\dots- t_{\abs{\gamma}}}\,dt_1\dots\,dt_{\abs{\gamma}}\\
&=\sum_{\gamma \in \Gamma^{sa}_{H}(v, w)} k^{\abs{\gamma}} \idotsint\limits_{\sum_i t_i \leq t^*/k} e^{ -\frac{ t^*/k }{ 1- e^{-t^*/k}} C_H(\gamma, kt_1, \dots, k t_{\abs{\gamma}})-t_1-\dots- t_{\abs{\gamma}}}\,dt_1\dots\,dt_{\abs{\gamma}}\\
&=\sum_{\gamma \in \Gamma^{sa}_{H}(v, w)} \idotsint\limits_{\sum_i t_i \leq t^*} e^{ -\frac{ t^*/k }{ 1- e^{-t^*/k}} C_H(\gamma, t_1, \dots, t_{\abs{\gamma}})-t_1/k-\dots- t_{\abs{\gamma}}/k}\,dt_1\dots\,dt_{\abs{\gamma}}.
\end{align*}
Letting $k\rightarrow\infty$, we obtain the final bound
\begin{equation}
\mathbb{P}\left(T_H(v, w) \leq t^*\right) \geq \sum_{\gamma\in\Gamma^{sa}_H(v, w)}\idotsint\limits_{\sum_i t_i \leq t^*} e^{-C_H(\gamma, t_1, \dots, t_{\abs{\gamma}})}\,dt_1\dots\,dt_{\abs{\gamma}}.
\end{equation}
The proposition follows from Proposition \ref{prop:infprob} by identifying the right-hand side in terms of $\{X_t\}_{t=0}^{t^*}$.
\end{proof}

Pick any $G, v$ and $w$ such that $t^*_G(v, w)< \infty$. Below, we let $\{X_t\}_{t=0}^{t^*}$ be the conditioned random walk on $G^n$ from $\bar{v}$ to $\bar{w}$ in time $t^*=t^*_G(v, w)$. In accordance with the above proposition, we let $L$ denote the number of steps taken by $X_t$, let $T_1, \dots, T_L$ be the times of these steps, and let $T_0=0$. We denote the coordinates of $X_t$ by $X_t^1, \dots, X_t^n$. Recall that each coordinate is an independent copy of $\{Y_t\}_{t=0}^{t^*}$, the conditioned random walk from $v$ to $w$ on $G$ in time $t^*$. We begin by making the following observation about $Y_t$:
\begin{lemma}\label{lemma:spread}
For $s, t\geq 0$ such that $s+t\leq t^*$ we have that
\begin{equation}
\mathbb{P}\left( Y_s\neq Y_{s+t}\right) = \Omega(t),
\end{equation}
hence  $\mathbb{P}\left( Y_s=Y_{s+t}\right) \leq e^{-\Omega(t)}$.
\end{lemma}
\begin{proof}
Let $v=v_0, v_1, \dots, v_l=w$ be a shortest path from $v$ to $w$ in $G$. Then
\begin{align*}
\mathbb{P}\left( Y_s \neq Y_{s+t} \right) &\geq \sum_{k=0}^{l-1} \mathbb{P}\left( Y_s = v_k, Y_{s+t}=v_{k+1}\right)\\
&= \sum_{k=0}^{l-1} m_G(v, v_k, s) m_G(v_k, v_{k+1}, t) m_G(v_{k+1}, v_l, t^*-s-t)\\
&\geq \sum_{k=0}^{l-1} \frac{s^{k}}{k!} t \frac{(t^*-s-t)^{l-k-1}}{(l-k-1)!}\\
&= \frac{t (t^*-t)^{l-1}}{(l-1)!}.
\end{align*}
We further have
\begin{align*}
\mathbb{P}(Y_s\neq Y_{s+t}) &\geq \mathbb{P}\left(Y_s=v, Y_{s+t}=w\right)\\
&= m_G(v, v, s) m_G(v, w, t) m_G(w, w, t^*-s-t)\\
&\geq m_G(v, w, t).
\end{align*}
The lemma follows by using the first bound for, say, $t\leq t^*/2$, and the second for $t>t^*/2$.
\end{proof}

Recall that $\Gamma_H(v, w)$ and $\Gamma^{sa}_H(v, w)$ denote the set of all paths from $v$ to $w$ in $H$ and the set of self-avoiding such paths respectively. 
\begin{proposition}\label{prop:selfavoid}
We have
\begin{equation}\label{eq:selfavoid}
\sum_{\gamma \in \Gamma^{sa}_{G^n}(\bar{v}, \bar{w})} \frac{(t^*)^{\abs{\gamma}}}{\abs{\gamma}!} = \mathbb{P}\left(X \text{ self-avoiding}\right) \geq \left(1-O(n^{-1})\right)e^{-(\Delta_o t^*)^2}
\end{equation}
where $\Delta_o=\Delta_o(G)$ denotes the maximal out-degree of any vertex in $G$.
\end{proposition}
\begin{proof}
Note that the first equality in \eqref{eq:selfavoid} is a direct consequence of \eqref{eq:probtracepath} and Proposition \ref{prop:multiplicativity}.

We say that a path given by the vertex sequence $v_0, v_1, \dots, v_l$ is almost self-avoiding if $v_i=v_j$ only if $\abs{i-j} \leq 2$. We will start by showing that the probability that $X_t$ is almost self-avoiding tends to one as $n\rightarrow\infty$.

Let $\bar{v}=v_0, v_1, \dots, v_L=\bar{w}$ denote the vertices along the path traced by $X_t$. Suppose that this path is not almost self-avoiding, that is there exist $i, j$ satisfying $j\geq i+3$ such that $v_j=v_i$. As $G$ contains no loops, there are only three possible ways this can occur:
\begin{itemize}
\item The random walker changes the same coordinate at times $T_{i+1}$ and $T_j$, and this coordinate has been changed at least once more during $(T_{i+1}, T_j)$. All other coordinates are the same at times $T_{i+1}$ and $T_j$. 
\item The random walker changes the same coordinate at times $T_{i+1}$ and $T_j$, and there is another coordinate that has been changed at least twice during $(T_{i+1}, T_j)$. All other coordinates are the same at times $T_{i+1}$ and $T_j$.
\item The random walker changes different coordinates at times $T_{i+1}$ and $T_j$, both of which have changed at least once more during $(T_{i+1}, T_j)$. All other coordinates are the same at times $T_{i+1}$ and $T_j$.
\end{itemize}

We start by considering the expected number of occurrences of the first case. Clearly, this is the same as $n$ times the expected number of occurrences corresponding to the first coordinate. By Proposition \ref{prop:boundcondbyuncond}, the probability that the first coordinate changes during $[s, s+ds)$, during $[s+t, s+t+dt)$ and at least once more during $(s, s+t)$ for $s, t >0$ such that $s+t< t^*$ is $O(t)\,ds\,dt$. Hence, by Lemma \ref{lemma:spread}, we get the upper bound
\begin{align*}
&n \int_0^{t^*} \int_0^{t^*-t} \left(\prod_{k=2}^n \mathbb{P}\left(X^k_s=X^k_{s+t}\right)\right) O(t) \,ds\,dt\\
&\qquad \leq  n \int_0^{t^*} \int_0^{t^*-t} O(t) e^{-(n-1)\Omega(t) }\,ds\,dt\\
&\qquad \leq n \int_0^\infty O(t) e^{-(n-1)\Omega(t)}\,dt = O(n^{-1}).
\end{align*}
Analogously, both of the remaining cases result in bounds of the form
\begin{align*}
&n(n-1) \int_0^{t^*} \int_0^{t^*-t} \left( \prod_{k=3}^n \mathbb{P}\left( X^k_s=X^k_{s+t}\right)\right) O(t^2)\,ds\,dt\\
&\qquad \leq n(n-1) \int_0^{t^*} \int_0^{t^*-t}  e^{-(n-2)\Omega(t)} O(t^2) \,ds\,dt\\
&\qquad \leq n(n-1) \int_0^\infty  e^{-(n-2)\Omega(t) }O(t^2) \,dt = O(n^{-1}). 
\end{align*}
Hence the probability that $X_t$ is almost self-avoiding is $1-O(n^{-1})$, as desired.

Let $\Gamma^{asa}_{G^n}(\bar{v}, \bar{w})$ denote the set of almost self-avoiding paths from $\bar{v}$ to $\bar{w}$ in $G^n$. By the argument above together with \eqref{eq:probtracepath}, we already know that
\begin{equation}\label{eq:asaestimate}
\sum_{\gamma \in \Gamma^{asa}_{G^n}(\bar{v}, \bar{w})} \frac{ \left(t^*\right)^{\abs{\gamma}}}{\abs{\gamma}!} = \mathbb{P}\left( X \text{ almost self-avoiding}\right) = 1-O(n^{-1}).
\end{equation}
Now, if we represent paths as sequences of edges, then each not necessarily self-avoiding path $\gamma'\in \Gamma_{G^n}(\bar{v}, \bar{w})$ contains a self-avoiding path $\gamma \in \Gamma_{G^n}^{sa}(\bar{v}, \bar{w})$ as a subsequence. In particular, if $\gamma'$ is almost self-avoiding, then it can be constructed from $\gamma$ by, at each vertex, either doing nothing or inserting a detour of length $2$ before the subsequent edge in $\gamma$. As each detour of length $2$ in $G^n$ must change the same coordinate twice, there are at most $n \Delta_o^2$ possible detours at each location. It follows that the contribution to \eqref{eq:asaestimate} from extensions of $\gamma$ is at most
\begin{align*}
\sum_{k=0}^{\abs{\gamma}+1} {\abs{\gamma}+1 \choose k} (\Delta_o^2n)^{k}\frac{ \left(t^*\right)^{\abs{\gamma}+2k}}{(\abs{\gamma}+2k)!} &\leq \frac{(t^*)^{\abs{\gamma}}}{ \abs{\gamma}! }\sum_{k=0}^{ \infty} \frac{ (\Delta_o^2n)^k (t^*)^{2k}}{k! \abs{\gamma}^{k}}\\
&= \frac{(t^*)^{\abs{\gamma}}}{ \abs{\gamma}! } \exp\left( \frac{ \Delta_o^2n(t^*)^{2}}{\abs{\gamma}}\right)\\
&\leq \frac{ (t^*)^{\abs{\gamma}}}{\abs{\gamma}!} e^{(\Delta_o t^*)^2},
\end{align*}
where, in the last step, we use that the distance between $\bar{v}$ and $\bar{w}$ is at least $n$. Hence, by \eqref{eq:asaestimate} we have
\begin{equation}
1-O(n^{-1}) \leq 
e^{(\Delta_o t^*)^2}\sum_{\gamma \in \Gamma^{sa}_{G^n}(\bar{v}, \bar{w})} \frac{(t^*)^{\abs{\gamma}}}{\abs{\gamma}!},
\end{equation}
as desired.
\end{proof}

We now turn to the problem of showing that $C(X)$ is not too large on average. In doing so, it turns out to be useful to treat the terms in $C(X)$ separately depending on the size of $T_j-T_i$. Let $\varepsilon>0$ be sufficiently small. We define
\begin{align}
C_1(X) &= \sum_{0 < T_j-T_i\leq \varepsilon} m_{G^n}(X_{T_i}, X_{T_j}, T_j-T_i)\label{eq:C1}\\
C_2(X) &= \sum_{T_j-T_i > \varepsilon} m_{G^n}(X_{T_i}, X_{T_j}, T_j-T_i).\label{eq:C2}
\end{align}

\begin{proposition}\label{prop:barsmallt}
For any fixed sufficiently small $\varepsilon>0$, we have $\mathbb{E}C_1(X)=O(1)$.
\end{proposition}
\begin{proof}
For any $1 \leq i \leq L$, let $K_i$ denote the coordinate that is updated at time $T_i$. We have three types of terms in \eqref{eq:C1}: those where $i=0$, where $K_i=K_j$, and  where $K_i\neq K_j$.

We proceed in a similar manner to the proof of Proposition \ref{prop:selfavoid}. By Proposition \ref{prop:boundcondbyuncond}, the probability that the random walk changes a given coordinate, say the first one, for the first time during $[t, t+dt)$ is $O(1)\,dt$. Similarly, the probability that the coordinate changes during $[t, t+dt)$ and at least once more before this $O(t) \,dt$. Note that in case of the former, $v=X^1_0$ and $X^1_{t+dt}$ are different and hence $m_G(v, X^1_{t+dt}, t) = O(t)$ by Lemma \ref{lemma:mbounds}. By coordinate independence of $X$ and the multiplicativity of the generating function, it follows that, conditioned on one of these events, the expected value of $m_{G^n}(\bar{v}, X_{t+dt}, t)$ is $$O(t) \prod_{k=2}^n \mathbb{E}\left[ m_G(v, X^k_{t+dt}, t)\right]$$ and $$O(1) \prod_{k=2}^n \mathbb{E}\left[ m_G(v, X^k_{t+dt}, t)\right]$$ respectively. It  follows that
\begin{equation}
\begin{split}\label{eq:barianceest1}
\mathbb{E} \sum_{T_j \leq \varepsilon} m_{G^n}(\bar{v}, X_{T_j}, T_j) &= n \int_0^\varepsilon \left(O(t) \prod_{k=2}^n \mathbb{E}\left[ m_G(v, X^k_t, t) \right]\right) O(1) \,dt\\
&\qquad + n\int_0^\varepsilon \left(O(1) \prod_{k=2}^n \mathbb{E}\left[ m_G(v, X^k_t, t) \right]\right) O(t)\,dt.
\end{split}
\end{equation}

Turning to the case where $K_i=K_j$, again by Proposition \ref{prop:boundcondbyuncond}, the probability that a given coordinate changes during $[s, s+ds)$ and during $[s+t, s+t+dt)$ is $O(1)\,ds\,dt$. Hence, in the same way as above, we see that
\begin{equation}
\begin{split}\label{eq:barianceest2}
&\mathbb{E}\sum_{\substack{0< T_j-T_i \leq \varepsilon\\ K_i=K_j}} m_{G^n}(X_{T_i}, X_{T_j}, T_j-T_i)\\
&\qquad = n \int_0^\varepsilon \int_0^{t^*-t} \left( O(1) \prod_{k=2}^n \mathbb{E}\left[m_G(X^k_s, X^k_{s+t}, t)\right] \right) O(1)\,ds\,dt.
\end{split}
\end{equation}

As for the case where $K_i\neq K_j$, the probability that, say, the first coordinate changes during $[s, s+ds)$, the second coordinate changes during $[s+t, s+t+dt)$, and the second coordinate does not change during $[s, s+t)$ is $O(1)\,ds\,dt$. Similarly, the probability if the second coordinate is required to change during $[s, s+t)$ is $O(t)\,ds\,dt$. In the former case, $X^2_s$ and $X^2_{s+t+dt}$ are different, and hence $m_G(X^2_s, X^2_{s+t}, t) = O(t)$. It follows that
\begin{equation}
\begin{split}\label{eq:barianceest3}
&\mathbb{E}\sum_{\substack{0 < T_j-T_i < \varepsilon\\ K_i\neq K_j}} m_{G^n}(X_{T_i}, X_{T_j}, T_j-T_i)\\
&\qquad= n(n-1) \int_0^\varepsilon \int_0^{t^*-t} \left( O(t) \prod_{k=3}^n \mathbb{E}\left[m_G(X^k_s, X^k_{s+t}, t)\right] \right) O(1) \,ds\,dt\\
&\qquad + n(n-1) \int_0^\varepsilon \int_0^{t^*-t} \left( O(1) \prod_{k=3}^n \mathbb{E}\left[m_G(X^k_s, X^k_{s+t}, t)\right] \right) O(t)\,ds\,dt.
\end{split}
\end{equation}

It remains to consider the expected value of $m_G(X^k_{s}, X^k_{s+t}, t)$. By Lemma \ref{lemma:mbounds}, we have
\begin{align*}
&\mathbb{E}\left[ m_G(X^k_s, X^k_{s+t}, t)\right] = \mathbb{P}\left( X^k_s = X^k_{s+t}\right) \left(1+O(t^2)\right) + \mathbb{P}\left( X^k_s \neq X^k_{s+t}\right) O(t)\\
&\quad =1 + O(t^2)-\mathbb{P}\left( X^k_s \neq X^k_{s+t}\right) \left( 1 - O(t)\right).
\end{align*}
It follows by Lemma \ref{lemma:spread} that for $\varepsilon>0$ sufficiently small we have $$\mathbb{E}[m_G( X^k_{s}, X^k_{s+t}, t)] = e^{-\Omega(t)},$$
for any $0\leq t \leq \varepsilon$. Plugging this into the bounds in \eqref{eq:barianceest1}, \eqref{eq:barianceest2} and \eqref{eq:barianceest3}, we see that the integrals evaluate to $O(n^{-1})$, $O(1)$ and $O(1)$ respectively.
\end{proof}

\begin{proposition}\label{prop:moderatebariance}
Suppose $f_G^{vw}(s, t) \leq -\Omega(t(t^*-t))$. Then for any $p<1$, there exists a constant $M$, not depending on $n$, such that
\begin{equation}
\mathbb{P}\left( C(X) \leq M \right) \geq p-o(1).
\end{equation}
\end{proposition}
\begin{proof}
By Proposition \ref{prop:barsmallt}, it only remains to check $C_2(X)$. Following the argument in the proof of that proposition, it follows that for any fixed $\alpha>0$,
\begin{equation}
\begin{split}\label{eq:malphaest}
&\mathbb{E} \sum_{T_j-T_i>\varepsilon} m_{G^n}(X_{T_i}, X_{T_j}, T_j-T_i)^{\alpha}\\
&\qquad = n \int_\varepsilon^{t^*} \left( O(1) \prod_{k=2}^n \mathbb{E}\left[ m_G(X^k_0, X^k_t, t)^\alpha\right]\right) O(1) \,dt\\
&\qquad + n \int_\varepsilon^{t^*} \int_{0}^{t^*-t} \left( O(1) \prod_{k=2}^n \mathbb{E}\left[ m_G(X^k_s, X^k_{s+t}, t)^\alpha\right]\right) O(1)\,ds\,dt\\
&\qquad + n^2\int_\varepsilon^{t^*} \int_{0}^{t^*-t} \left( O(1) \prod_{k=3}^n \mathbb{E}\left[ m_G(X^k_s, X^k_{s+t}, t)^\alpha\right]\right) O(1)\,ds\,dt.
\end{split}
\end{equation}

By Propositions \ref{prop:convisprob} and \ref{prop:unifconv}, we have
\begin{equation}\label{eq:mtaylor}
\begin{split}
&\mathbb{E}\left[ m_G(X^k_s, X^k_{s+t}, t)^\alpha\right]\\
&\qquad= \sum_{x, y\in G} m_G(v, x, s) m_G(x, y, t)^{1+\alpha} m_G(y, w, t^*-s-t)\\
&\qquad=1+\alpha f_G^{vw}(s, t) + O\left( (t^*-t)\alpha^2\right).
\end{split}
\end{equation}
Under the assumption that $f_G^{vw}(s, t) = -\Omega(t(t^*-t))$, it follows  that for sufficently small $\alpha>0$ we have $\mathbb{E}\left[ m_G(X^k_s, X^k_{s+t}, t)^\alpha\right] \leq e^{-\Omega((t^*-t)\alpha)}$ for $\varepsilon \leq t \leq t^*$. Plugging this into \eqref{eq:malphaest}, it follows that
\begin{equation}
\mathbb{E} \sum_{T_j-T_i>\varepsilon} m_{G^n}(X_{T_i}, X_{T_j}, T_j-T_i)^{\alpha} = O(1)+O\left(n^{-1}\right)+O(1).
\end{equation}

The Proposition follows by Markov's inequality together with the elementary inequality
\begin{align*}
C_2(X) &= \sum_{T_j-T_i>\varepsilon} m_{G^n}(X_{T_i}, X_{T_j}, T_j-T_i)\\ &\leq \left(\sum_{T_j-T_i>\varepsilon} m_{G^n}(X_{T_i}, X_{T_j}, T_j-T_i)^{\alpha}\right)^{1/\alpha},
\end{align*}
which holds for any $0<\alpha \leq 1$.
\end{proof}

In Proposition \ref{prop:moderatebariance} we take $p\geq 1-e^{-(\Delta_o t^*)^2}/2$ so that by plugging Proposition \ref{prop:selfavoid} into the right-hand side of \eqref{eq:Tubrw} with $H=G^n$ we get
\begin{equation}\label{eq:posprobub}
\begin{split}
\mathbb{P}(T_{G^n}(\bar{v}, \bar{w}) \leq t^*) &\geq \mathbb{E}[\mathbbm{1}_{X\text{ s.a.}}e^{-C(X)} \mathbbm{1}_{C(X) \leq M }]\\ &\geq e^{-M}e^{-(\Delta_o t^*)^2}/2 - O(n^{-1}),
\end{split}
\end{equation}
which is bounded away from zero uniformly in $n$. As we shall see in the next subsection, this statement is sufficient to prove convergence by more direct means, using the structure of power graphs.


\subsection{An asymptotically almost sure upper bound}\label{sec:recineq}

The argument given here is a variation of the argument given in Section 5 of \cite{M16} which proves an asymptotically almost sure upper bound on the first-passage time between antipodal vertices in the hypercube. Let $G$ be any graph, and let $v$ and $w$ be distinct vertices in $G$. We start by deriving a recursive relation on $T_{G^n}(\bar{v}, \bar{w})$. Assume $n\geq 3$. For each $1\leq i \leq n$, let $\bar{v}^i$ denote the vertex in $G^n$ whose $i$:th coordinate is $w$ and all other coordinates are $v$. Similarly, let $\bar{w}^i$ be the vertex in $G^n$ whose $i$:th coordinate is $v$ and all other coordinates are $w$.

Let $\gamma$ be a path from $v$ to $w$ in $G$ of minimal length. For each $1\leq i \leq n$ let $\gamma_i$ and $\gamma'_i$ be the paths from $\bar{v}$ to $\bar{v}^i$  and from $\bar{w}^i$ to $\bar{w}$ respectively whose $i$:th projections are $\gamma$ and whose $j$:th projections for all $j\neq i$ are trivial. Given the passage times of these paths, pick distinct indices $i_1$ and $i_2$ that minimize
 $$T_{G^n}(\gamma_{i_1})+T_{G^n}(\gamma'_{i_1}) + T_{G^n}(\gamma_{i_2})+T_{G^n}(\gamma'_{i_2}).$$ Without delving into any technical calculations, it is not too hard to convince oneself that the expected value of this sum is $o(1)$, as each $T_{G^n}(\gamma_i)$ and $T_{G^n}(\gamma_i')$ is an independent sum of a fixed number of independent $\Exp(1)$ random variables.
 
Given $i_1$ and $i_2$, we let $H_1$ equal the subgraph of $G^n$ consisting of all vertices where the $i_1$:th coordinate is $w$ and the $i_2$:th coordinate is $v$. Similarly, let $H_2$ be the induced subgraph where the $i_1$:th coordinate is $v$ and the $i_2$:th coordinate is $w$. Some observations are in order:
\begin{itemize}
\item $H_1$ and $H_2$ are vertex disjoint.
\item $H_1$ and $H_2$ are both isomorphic to  $G^{n-2}$.
\item The only vertices that $H_{k}$, for $k=1, 2$, can have in common with any path $\gamma_i$ or $\gamma'_i$ are $\bar{v}^{i_{k}}$ and $\bar{w}^{i_{3-k}}$. In particular, $H_k$ has no edge in common with $\gamma_i$ or $\gamma'_i$ since $n\geq 3$.
\end{itemize}

It follows that $T_{H_1}(\bar{v}^{i_1}, \bar{w}^{i_2})$ and $T_{H_2}(\bar{v}^{i_2}, \bar{w}^{i_1})$ are independent and have the same distribution as $T_{G^{n-2}}(\bar{v}, \bar{w})$. By bounding $T_{G^n}(\bar{v}, \bar{w})$ by the minimum of
$$T_{G^n}(\gamma_{i_1}) + T_{H_1}(\bar{v}^{i_1}, \bar{w}^{i_2}) + T_{G^n}(\gamma'_{i_2})$$
and
$$T_{G^n}(\gamma_{i_2}) + T_{H_1}(\bar{v}^{i_2}, \bar{w}^{i_1}) + T_{G^n}(\gamma'_{i_1})$$
we have shown the following:
\begin{proposition}\label{prop:recineq}
Assume $n\geq 3$. There exists a non-negative random variable $\xi_n$ with expected value $o(1)$ such that $T_{G^n}(\bar{v}, \bar{w})$ is stochastically dominated by $\xi_n$ plus the minimum of two independent copies of $T_{G^{n-2}}(\bar{v}, \bar{w})$. \qed
\end{proposition}

\begin{lemma}\label{lemma:minofiid}
Let $T$ and $T'$ be independent identically distributed random variables. Then for any real value $C$ and any $0 \leq p \leq \mathbb{P}\left( T \leq C\right)$,
$$\mathbb{E} \min( T, T') \leq p C + (1-p) \mathbb{E} T .$$
\end{lemma}
\begin{proof}
The case where $p=0$ is obvious. For $p=\mathbb{P}\left( T \leq C\right)$, the inequality follows by taking the expected value of the bound $ \min( T, T') \leq C \mathbbm{1}_{T\leq C} + T' \mathbbm{1}_{T>C}.$ The general case is a convex combination of these.
\end{proof}

\begin{proposition}\label{prop:recineq}
Let $G$ be a graph with vertices $v$ and $w$. Suppose there exists a constant $C>0$ such that $$\liminf\limits_{n\rightarrow\infty} \mathbb{P}\left( T_{G^n}(\bar{v}, \bar{w}) \leq C \right) > 0,$$
then $$\mathbb{E}T_{G^n}(\bar{v}, \bar{w}) \leq C + o(1),$$
as $n\rightarrow\infty$.
\end{proposition}
\begin{proof}
Fix an $N$ and $p>0$ such that $\mathbb{P}\left( T_{G^n}(\bar{v}, \bar{w}) \leq C \right) \geq p$ whenever $n\geq N$. It follows by Proposition \ref{prop:recineq} and Lemma \ref{lemma:minofiid} that for any $n\geq N+2$ we have
\begin{equation}
\mathbb{E} T_{G^n}(\bar{v}, \bar{w}) - C  \leq o(1) + (1-p) \left( \mathbb{E} T_{G^{n-2}}(\bar{v}, \bar{w}) - C\right).
\end{equation}
As $\mathbb{E} T_{G^N}(\bar{v}, \bar{w})$ and $\mathbb{E} T_{G^{N+1}}(\bar{v}, \bar{w})$ are finite, this recursion implies that $\mathbb{E} T_{G^n}(\bar{v}, \bar{w}) - C \leq o(1)$, as desired.
\end{proof}

In particular, if the hypothesis in Proposition \ref{prop:recineq} holds for $C=t^*$, or at least $C=t^*+\varepsilon$ for all $\varepsilon>0$, this implies that $\mathbb{E}T_{G^n}(\bar{v}, \bar{w}) \leq t^*+o(1)$. As noted in the beginning of the Section \ref{sec:sharp}, this implies convergence in probability and $L^1$, as desired.



\subsection{Generalization to weighted graphs}\label{sec:weightedgraphs}
It remains to prove that $T_{G^n}(\bar{v}, \bar{w}) \leq t^*+o(1)$ with probability bounded away from zero as $n\rightarrow\infty$ in the case where $f_G^{vw}(s, t)\leq 0$ everywhere, but $f_G^{vw}(s, t) \not\leq -\Omega(t(t^*-t))$. In order to do so, we need to generalize the argument in Subsection \ref{sec:weakbound}. Below, by a weighted graph we mean a graph $H=(V, E)$ where each edge $e$ is assigned a positive weight $\lambda(e)$, called its intensity. The Cartesian product of two weighted graphs is the weighted version of $H_1 \square H_2$ where the edges $(e, v)$ and $(v, e)$ are given the same intensities as the edge $e$ in the respective factors. First-passage percolation is defined on such graphs by independently assigning the passage time of each edge $e$ according to $\Exp(\lambda(e))$ random variables.

The generating function of a weighted graph $H$ is defined as
\begin{equation}
m_H(v, w, t) = \sum_{\gamma \in \Gamma(v, w)} \left(\prod_{e \in \gamma}\lambda(e)\right) \frac{t^{\abs{\gamma}}}{\abs{\gamma}!}, 
\end{equation}
where the factor $\lambda(e)$ appears the same number of times that $\gamma$ traverses $e$. One can observe that the generating function for weighted graphs has the same basic properties as for unweighted graphs:
\begin{proposition}Let $H$ be a weighted graph, and let $v, w$ be vertices in $H$.
\begin{enumerate}[label=\roman*)]
\item For any $t\geq 0$, we have
\begin{equation*}
\mathbb{P}\left( T_H(v, w) \leq t \right) \leq m_H(v, w, t).
\end{equation*}
\item For any $s, t\geq 0$,
\begin{equation*}
\sum_{x\in H} m_H(v, x, s)m_H(x, w, t) = m_H(v, w, s+t).
\end{equation*}
\item For any weighted graphs $H_1$ and $H_2$, and any vertices $v_1, w_1 \in H_1$ and $v_2, w_2\in H_2$, we have
\begin{equation*}
m_{H_1 \square H_2}\left( (v_1, v_2), (w_1, w_2), t\right) = m_{H_1}(v_1, w_1, t) m_{H_2}(v_2, w_2, t). 
\end{equation*}
\end{enumerate}
\end{proposition}
The proofs of these properties are almost identical to the unweighted case, and will not be repeated here.

Let $G$ be a graph with distinct vertices $v$ and $w$. We assume that $f_G^{vw}(s, t) \leq 0$ everywhere. We will consider $G$ as a weighted graph with constant intensity $1$. Let $l$ be the distance from $v$ to $w$ in $G$ and let $P$ be a graph with vertex set of the form $\{v_0, v_1, \dots, v_l\}$ where there is a directed edge from $v_i$ to $v_{i+1}$ for each $0\leq i < l$. To simplify notation below we also denote the end-points of $P$ by $v$ and $w$ respectively. We make $P$ into a weighted graph by assigning its edges a common intensity $\lambda$ such that $t^*_G(v, w) = t^*_P(v, w)$, that is $\lambda = \sqrt[l]{l!}/t^*_G(v, w)$. The key idea to prove the remaining part of Theorem \ref{thm:main} is to consider the first-passage time from $\bar{v}$ to $\bar{w}$ in graphs of the form $G^{n-k}\square P^{k}$. Note that $G^{n-k}\square P^k$ can be seen as a subgraph of $G^n$, though as the intensities of edges going in the last $k$ directions may be higher in the former than the latter it is not clear if first-passage times in one should dominate those in the other. 

The definition of the conditioned random walk on a graph can be naturally extended to a weighted graph $H$ by letting $\Delta_o=\Delta_o(H)$ denote the maximal net intensity of outgoing edges from any vertex $x\in H$ and, whenever the walker tries to take a step, picking each outgoing edge $e$ with probability $\lambda(e)/\Delta_o$. Again, we can note that if $H$ is a Cartesian product graph, then the projections of a conditioned random walk on $H$ onto each factor are independent conditioned random walks. Furthermore, for any path $\gamma \in \Gamma_{G^n}(\bar{v}, \bar{w})$, the probability that the walker traces $\gamma$ and takes its steps during $[t_1, t_1+dt_1)$, $[t_2, t_2+dt_2)\dots$ for $0< t_1 < \dots < t_{\abs{\gamma}} < t^*$ is $\left( \prod_{e\in \gamma} \lambda(e) \right)\,dt_1\dots\,dt_{\abs{\gamma}}$. As before, we let $L$ denote the number of steps taken by the walker, let $T_0=0$ and $T_i$ for any $1\leq i \leq L$ denote the time of the $i$:th step.

Using this generalization, one can see that the proof of Proposition \ref{prop:Tubrw} follows through also for weighted graphs with bounded out-intensity if one revises the upper bound on the probability that any fixed path is a shortcut, and is mindful of the different intensities when integrating over $t_1, \dots, t_{\abs{\gamma}}$. 

\begin{lemma}\label{lemma:Pconvex}
We have
\begin{equation}
\begin{split}
f_P^{vw}(s, t) &:= \sum_{x,y\in P} m_P(v, x, s) m_P(x, y, t) m_P(y, w, t^*-s-t) \ln\left(m_P(x, y, t)\right)\\
&\leq -\Omega\left( t(t^*-t) \right).
\end{split}
\end{equation}
\end{lemma}
This can be proven by a couple of lines of straight-forward but messy calculations by showing that $f_P^{vw}(s, t)$ does not depend on $s$ and is convex in $t$. Alternatively, this can be seen as an implication of Proposition \ref{prop:charf} and Remark \ref{rem:strongconvexity}, by considering $P$ as a subgraph of the doubly infinite directed chain $\vec{\mathbb{Z}}$.

\begin{proposition}\label{prop:weightedbarestimates}
Let $k=k_n$ be a sequence of integers such that $0\leq k_n \leq n$ and $k_n/n\rightarrow r \in (0, 1)$ as $n\rightarrow\infty$. Let $\{X_t\}_{t=0}^{t^*}$ be the conditioned random walk from $\bar{v}$ to $\bar{w}$ on $G^{n-k}\square P^k$ in time $t^*$. Then, for any sufficiently small constant $\varepsilon>0$ we have:
\begin{enumerate}[label=\roman*)]
\item With probability at least $\left(1-O\left(n^{-1}\right)\right) e^{-(\Delta_o(G) t^*)^2}$, $X$ is self-avoiding.
\item $$\mathbb{E}\left[\sum_{T_i<T_j<T_i+\varepsilon} m_{G^{n-k}\square P^{k}}(X_{T_i}, X_{T_j}, T_j-T_i)\right] = O(1).$$
\item For any sufficiently small fixed $\alpha>0$,
$$\mathbb{E}\left[\sum_{T_j \geq T_i+\varepsilon} m_{G^{n-k}\square P^{k}}(X_{T_i}, X_{T_j}, T_j-T_i)^\alpha\right] = O(1).$$ 
\end{enumerate}
\end{proposition} 
\begin{proof}
$i)$ As $P$ does not contain cycles, it is easy to see that $X$ is self-avoiding if the conditioned random walk on $G^{n-k}$ formed by the first $n-k$ coordinates is so. The statement follows by Proposition \ref{prop:selfavoid}. $ii)$ This can be shown in the same way as Proposition \ref{prop:barsmallt}.  $iii)$ We can bound this expression analogously to \eqref{eq:malphaest}, except that now about $rn$ factors of the form $\mathbb{E}\left[ m_G(X^i_s, X^i_{s+t}, t)^\alpha\right]$ are replaced by $\mathbb{E}\left[ m_P(X^i_s, X^i_{s+t}, t)^\alpha\right]$. The idea is that these factors will make the integrals stay small even when we do not have $f_G^{vw}(s, t) = -\Omega\left( t(t^*-t)\right)$.

As $f_G^{vw}(s, t) \leq 0$, we have by \eqref{eq:mtaylor} that $E\left[ m_G(X^i_s, X^i_{s+t}, t)^\alpha\right] \leq e^{O\left( (t^*-t)\alpha^2 \right)}$ for any $1\leq i \leq n-k$. Similarly, for any $n-k+1 \leq i \leq n$, we have
\begin{align*}
\mathbb{E}\left[ m_P(X^i_s, X^i_{s+t}, t)^\alpha \right] &= \sum_{x, y \in P} m_P(v, x, s) m_P(x, y, t)^{1+\alpha} m_P(y, w, t^*-s-t)\\
&= 1+\alpha f_P^{vw}(s, t) + O\left( (t^*-t)\alpha^2\right).
\end{align*}
By Lemma \ref{lemma:Pconvex}, this expression is at most $e^{-\Omega\left((t^*-t)\alpha\right) + O\left( (t^*-t)\alpha^2\right)}$ for $\varepsilon \leq t \leq t^*$. Hence, we have
\begin{align*}
&\mathbb{E}\left[\sum_{T_j \geq T_i+\varepsilon} m_{G^{n-k}\square P^{k}}(X_{T_i}, X_{T_j}, T_j-T_i)^\alpha\right]\\
&\qquad \leq n\int_{\varepsilon}^{t^*} O(1) e^{-rn\,\Omega\left( (t^*-t) \alpha\right) + n\,O\left( (t^*-t)\alpha^2\right) }\,dt\\
&\qquad + n^2 \int_{\varepsilon}^{t^*}\int_0^{t^*-t} O(1) e^{-rn\,\Omega\left( (t^*-t) \alpha\right) + n\,O\left( (t^*-t)\alpha^2\right)}\,ds\,dt,
\end{align*}
which is order one provided $\alpha > 0$ is sufficiently small so that the exponents in the respective integrands are dominated by $-rn\Omega\left( (t^*-t) \alpha\right)$.
\end{proof}

We have the following immediate consequence of Propositions \ref{prop:Tubrw} and \ref{prop:weightedbarestimates}:
\begin{proposition}\label{prop:fppperturbedGn}
For any $k=k_n$ such that $k/n\rightarrow r\in(0, 1)$, we have
\begin{equation}
\liminf_{n\rightarrow\infty}\mathbb{P}\left( T_{G^{n-k}\square P^k}(\bar{v}, \bar{w}) \leq t^*\right) > 0.
\end{equation}
\qed
\end{proposition}

\begin{lemma}\label{lemma:orderstats}
For any integers $0 \leq K \leq N$, any $t>0$ and any subset $I\subseteq \{1, \dots, N\}$ of size $K$ we have
\begin{equation}\label{eq:orderstats}
\frac{N!}{t^N}\idotsint\limits_{\substack{t_1+\dots+t_N \leq t\\ t_1, \dots, t_N \geq 0 }} \mathbbm{1}_{\sum_{i\in I} t_i> 2t\frac{K}{N}} \,dt_1\dots\,dt_N \leq \frac{2}{K}.
\end{equation}
\end{lemma}
\begin{proof}
Note that the statement is trivially true unless $2K \leq N$. Let us, without loss of generality assume that $I_\gamma = \{1, 2, \dots, \abs{I_\gamma}\}$. By the substitution $s_i = \sum_{j=1}^i t_i/t$, the left hand side of \eqref{eq:orderstats} becomes
\begin{equation}
N! \idotsint_{0 \leq s_1 \leq \dots \leq s_N \leq 1} \mathbbm{1}_{s_K > 2 \frac{K}{N}} \,ds_1\dots \,ds_N.
\end{equation}
We can interpret this integral as the probability that, when given $N$ independent $\operatorname{U}(0, 1)$ random variables, less than $K$ of them have values less than $2\frac{K}{N}$. The bound follows by Chebyshev's inequality.
\end{proof}

\begin{proposition}\label{prop:perturbationsmall}
Suppose $k/n \rightarrow r\in (0, 1)$. With probability tending to one as $n\rightarrow\infty$, there are no self-avoiding paths $\gamma$ from $\bar{v}$ to $\bar{w}$ in $G^{n-k}\square P^k$ that satisfy both of:
\begin{itemize}
\item $T_{G^{n-k}\square P^k}(\gamma) \leq t^*$
\item the total passage time of edges going in the last $k$ directions is at least $2t^* \frac{k}{n}$.
\end{itemize}
\end{proposition}
\begin{proof}
For each path $\gamma \in \Gamma_{G^{n-k}\square P^k}(\bar{v}, \bar{w})$, let $I_\gamma$ denote the indices of edges in $\gamma$ going in the last $k$ directions. Note that $\abs{I_\gamma}=k\operatorname{dist}_G(v, w)$ and $\abs{\gamma}\geq n \operatorname{dist}_G(v, w)$ for any such $\gamma$. Using Lemma \ref{lemma:orderstats}, the expected number of paths with the proposed two properties is at most
\begin{align*}
&\sum_{\gamma \in \Gamma_{G^{n-k}\square P^k}(\bar{v}, \bar{w})} \left( \prod_{e\in \gamma} \lambda(e)\right) \idotsint\limits_{\substack{t_1+\dots+t_{\abs{\gamma}}\leq t^*\\ t_1, \dots, t_{\abs{\gamma}}\geq 0 } } \mathbbm{1}_{\sum_{i\in I_\gamma} t_i > 2 t^* \frac{k \operatorname{dist}_G(v, w)}{\abs{\gamma}}} \,dt_1\dots\,dt_{\abs{\gamma}}\\
&\qquad\leq \sum_{\gamma \in \Gamma_{G^{n-k}\square P^k}(\bar{v}, \bar{w})} \left( \prod_{e\in \gamma} \lambda(e)\right) \frac{2}{k\operatorname{dist}_G(v, w)} \frac{ (t^*)^{\abs{\gamma}}}{\abs{\gamma}!}\\
&\qquad = \frac{2}{k\operatorname{dist}_G(v, w)} m_{G^{n-k}\square P^k}(\bar{v}, \bar{w}, t^*)= \frac{2}{k\operatorname{dist}_G(v, w)},
\end{align*}
which tends to $0$ as $n\rightarrow\infty$.
\end{proof}
\begin{proposition}
For any $\varepsilon > 0$, we have
\begin{equation}
\liminf_{n\rightarrow\infty} \mathbb{P}\left( T_{G^n}(\bar{v}, \bar{w}) \leq t^*+\varepsilon\right) > 0.
\end{equation}
\end{proposition}
\begin{proof}
Let $k=k_n$ be a sequence of integers such that $k/n\rightarrow r\in (0, 1)$. Fix a suitable mapping of $G^{n-k}\square P^k$ into $G^n$. We can couple $T_{G^n}(\bar{v}, \bar{w})$ and $T_{G^{n-k}\square P^k }(\bar{v}, \bar{w})$ by using the same edge passage times, but weighting the passage times of edges in the last $k$ directions a factor $\lambda$ higher in the former than in the latter. By combining Propositions \ref{prop:fppperturbedGn} and \ref{prop:perturbationsmall}, it follows that
\begin{equation}
\liminf_{n\rightarrow\infty} \mathbb{P}\left( T_{G^n}(\bar{v}, \bar{w}) \leq t^*\left(1+2(\lambda-1)\frac{k}{n}\right)\right) > 0.
\end{equation}
The proposition follows by taking $r < \frac{\varepsilon}{2(\lambda-1)t^*}$.
\end{proof}

Given this proposition, we can proceed as in Subsection \ref{sec:recineq} to prove convergence in probability and $L^1$ as $n\rightarrow\infty$, which finishes the proof of Theorem \ref{thm:main} in the case of standard exponential weights..

\section{Proof of Theorem \ref{thm:gendist}}\label{sec:gendist}
To simplify notation we always assume that $\rho=1$. For $F\in\mathcal{C}(1)$, we couple $T_{H_n}(\cdot, \cdot)$ to $T^F_{H_n}(\cdot, \cdot)$ by letting $T^F_{H_n}(e) = h( T_{H_n}(e))$ for all $e\in H_n$ for some suitable increasing function $h$. In particular, we can take
\begin{equation}
h(t) = \inf\{x \geq 0: F(x) \geq 1-e^{-t}\}.
\end{equation}
One can check that $\lim_{t\rightarrow 0} h(t)/t = 1$. Throughout this section $\varepsilon>0$ will denote an arbitrarily small number, and $\delta>0$ a number such that $(1-\varepsilon)t \leq h(t) \leq (1+\varepsilon)t$ for all $0\leq t \leq \delta$.

An important tool in the coupling arguments below is to further consider first-passage percolation on $H_n$ with independent $\Exp(1+\lambda)$ weights. We denote this by $T^\lambda_{H_n}(\cdot, \cdot)$. This is coupled to $T_{H_n}(\cdot, \cdot)$ and $T^F_{H_n}(\cdot, \cdot)$ by, for each edge $e\in H_n$, generating an independent $\Exp(\lambda)$ random variable $\tau_e$, and letting $T^{\lambda}_{H_n}(e)=\min(T_{H_n}(e), \tau_e)$. Note that $T^\lambda_{H_n}(\cdot, \cdot)$ has the same distribution as $\frac{1}{1+\lambda} T_{H_n}(\cdot, \cdot)$.

\begin{lemma}\label{lemma:Flower}
For any fixed $t, \varepsilon, \lambda>0$, we have that if $$\mathbb{P}(T_{H_n}(v_n, w_n) < t) \rightarrow 0 \text{ as }n\rightarrow\infty,$$ then $$\mathbb{P}(T^F_{H_n}(v_n, w_n) < \frac{1-\varepsilon}{1+\lambda}t) \rightarrow 0\text{ as }n\rightarrow\infty.$$
\end{lemma}
\begin{proof}
Let $s=\frac{1}{1+\lambda} t$. As
\begin{align*}
&\mathbb{P}(T_{H_n}(v_n, w_n) < t) = \mathbb{P}(T^{\lambda}_{H_n}(v_n, w_n) < s)\\
&\qquad \geq \mathbb{P}\left( T^\lambda_{H_n}(v_n, w_n) < s \middle\vert T^F_{H_n}(v_n, w_n) < (1-\varepsilon)s\right)\mathbb{P}(T^F_{H_n}(v_n, w_n) < (1-\varepsilon)s),
\end{align*}
it suffices to show that 
\begin{equation}
\mathbb{P}\left( T^\lambda_{H_n}(v_n, w_n) < s \middle\vert T^F_{H_n}(v_n, w_n) < (1-\varepsilon)s\right)
\end{equation}
is bounded away from $0$ as $n\rightarrow\infty$.

Condition on the edge passage times $\{T_{H_n}(e)\}_{e\in H_n}$ and $\{T^{F}_{H_n}(e)\}_{e\in H_n}$. Suppose there is a path $\gamma$ from $v_n$ to $w_n$ such that $T^F_{H_n}(\gamma) < (1-\varepsilon)s$. Then, for any $e\in \gamma$ such that $T_{H_n}^{F}(e) < h(\delta)$ we have
\begin{equation}
T_{H_n}^\lambda(e)  \leq T_{H_n}(e) \leq \frac{1}{1-\varepsilon} h(T_{H_n}(e)) = \frac{1}{1-\varepsilon} T_{H_n}^{F}(e).
\end{equation}
As at most $\frac{1-\varepsilon}{h(\delta)}s$ edges in $\gamma$ satisfy $T_{H_n}^{F}(e) \geq h(\delta)$, the probability that $\tau_e \leq \frac{h(\delta)}{1-\varepsilon}$ for all such edges is bounded away from zero. If this happens, we have $T_{H_n}^\lambda(e) \leq \frac{1}{1-\varepsilon} T_{H_n}^{F}(e)$ for all $e\in \gamma$, and hence $T_{H_n}^\lambda(v_n, w_n) \leq T_{H_n}^\lambda(\gamma) < s$, as desired.
\end{proof}

As $H_n$ has bounded degree, there are almost surely $\Exp(1)$-geodesics from $v_n$ to $w_n$. One can for instance see this by noting that the number of self-avoiding paths from $v_n$ to $w_n$ with passage time at most $t$ is almost surely finite for any $t\geq 0$. Moreover, as the standard exponential distribution is continuous, almost surely no two self-avoiding paths have the same passage time, which implies that such a  geodesic is almost surely unique. Below, we let $\Gamma_n$ denote this unique $\Exp(1)$-geodesic.

\begin{lemma} Let $\varepsilon$ and $\delta$ be as above. Assume $T_{H_n}(v_n, w_n)\rightarrow t$ in probability as $n\rightarrow\infty$. Then, as $n\rightarrow\infty$
\begin{enumerate}[label=\roman*)]
\item $\mathbb{P}\left( T_{H_n}(e) \leq \delta\text{ for all }e\in \Gamma_n\right)\rightarrow 1$.
\item $T^F_{H_n}(v_n, w_n) \rightarrow t$ in probability,
\item Suppose $\int_0^\infty x\,dF(x) < \infty$. If $T_{H_n}(v_n, w_n)$ converges in $L^1$, then so does $T^F_{H_n}(v_n, w_n)$.
\end{enumerate}
\end{lemma}
\begin{proof}
$i)$ Condition on $\{T_{H_n}(e)\}_{e\in H_n}$. If there is an edge $e\in \Gamma_n$ such that $T_{H_n}(e) > \delta$ then with probability $1-e^{-\lambda\delta/2}>0$ we have $\tau_e\leq \frac{\delta}{2}$ and hence 
\begin{equation}\label{eq:smallweights}
T_{H_n}^\lambda(v_n, w_n) < T_{H_n}(v_n, w_n) - \frac{\delta}{2}.
\end{equation}
But as $T_{H_n}^\lambda(v_n, w_n) \xrightarrow{p} \frac{t}{1+\lambda}$ and $T_{H_n}(v_n, w_n) \xrightarrow{p} t$, the probability of \eqref{eq:smallweights} tends to $0$ assuming $\lambda$ is sufficiently small.

$ii)$ If $T_{H_n}(e) \leq \delta$ for all $e\in\Gamma_n$, then
\begin{equation}
\begin{split}
T^F_{H_n}(v_n, w_n) &\leq T^F_{H_n}(\Gamma_n) = \sum_{e\in\Gamma} h(T_{H_n}(e))\\ &\leq  \sum_{e\in\Gamma} (1+\varepsilon) T_{H_n}(e) = (1+\varepsilon) T_{H_n}(v_n, w_n).
\end{split}
\end{equation}
By $i)$ this occurs with probability tending to $1$. The corresponding lower bound follows from Lemma \ref{lemma:Flower}.

$iii)$ As $T^F_{H_n}(v_n, w_n)$ tends to $t$ in probability and is non-negative, it suffices to show that $\mathbb{E}T^F_{H_n}(v_n, w_n) \leq t+o(1)$. For any measurable map $\phi:[0, \infty) \rightarrow[0, \infty)$, let $T_n(\phi) = \sum_{e\in \Gamma_n} \phi(T_{H_n}(e))$. Clearly $T_{H_n}(v_n, w_n) = T_n(id)$ and $T^F_{H_n}(v_n, w_n) \leq T_n(h)$. We can rewrite the expectation of $T_n(\phi)$ as
\begin{equation}
\begin{split}
\mathbb{E} \sum_{e\in H_n} \mathbbm{1}_{e\in \Gamma_n} \phi(T_{H_n}(e)) &=  \sum_{e\in H_n} \int_0^\infty e^{-t} \phi(t) \mathbb{P}\left(e\in \Gamma_n \vert T_{H_n}(e) = t\right)\,dt\\
&= \int_0^\infty e^{-t} \phi(t) f_n(t)\,dt,
\end{split}
\end{equation}
where $f_n(t) = \sum_{e\in H_n} \mathbb{P}\left(e\in \Gamma_n \vert T_{H_n}(e) = t\right)$. As $\mathbb{E}T_n(id) = t+o(1) < \infty$, $f_n(t)$ is finite almost everywhere. Moreover, a simple coupling argument shows that $\mathbb{P}\left(e\in \Gamma_n \vert T_{H_n}(e) = t\right)$ is decreasing in $t$, hence so is $f_n(t)$.

Now,
\begin{equation}
\begin{split}
\mathbb{E} T_n(h) &= \int_0^\infty e^{-t} h(t) f_n(t)\,dt \\
&\leq (1+\varepsilon) \int_0^\delta e^{-t} t f_n(t)\,dt + f_n(\delta) \int_\delta^\infty e^{-t} h(t) \,dt\\
&\leq (1+\varepsilon) \int_0^\infty e^{-t} t f_n(t)\,dt + f_n(\delta)\int_0^\infty e^{-t} h(t)\,dt\\
&= (1+\varepsilon) \mathbb{E} T_n(id) + f_n(\delta)\int_0^\infty x\,dF(x).
\end{split}
\end{equation}
It remains to show that $f_n(\delta)\rightarrow 0$ as $n\rightarrow\infty$ for any $\delta>0$. Let $g(t)=\mathbbm{1}_{t\leq \delta/2} t$. By $i)$, and the fact that $T_n(id)\xrightarrow{p} t$, we know that $T_n(g) \geq t-o(1)$ with probability $1-o(1)$. In particular $\mathbb{E} T_n(g) \geq t-o(1)$. Hence $o(1) \geq \mathbb{E}T_n(id) - \mathbb{E}T_n(g) = \int_{\delta/2}^\infty e^ {-t} t f_n(t)\,dt \geq \int_{\delta/2}^\delta e^ {-t} t f_n(t)\,dt \geq f_n(\delta)\int_{\delta/2}^\delta e^{-t} t\,dt,$ as desired.
\end{proof}

\section*{Acknowledgements}
I am very grateful to my supervisor, Peter Hegarty, for his valuable comments and discussions. I would further like to thank the anonymous referees for their careful reading of the paper, and many thoughtful comments and suggestions.

\end{document}